\documentclass[12pt]{amsart}
\usepackage{amsmath}
\usepackage{amsfonts}
\usepackage{amsthm}
\usepackage{amssymb}
\usepackage{amscd}
\usepackage[all]{xy}
\usepackage{enumerate}

\keywords{Fibrations, Dualizing sheaf, Massey products, Fujita decompositions, Semistable fibrations, Local systems} 

\subjclass[2010]{14D06, 14C30, 14J40, 32G20}

\theoremstyle{plain}
\newtheorem{thm}{Theorem}[section]
\newtheorem{thml}{Theorem}

\newtheorem{prop}[thm]{Proposition}

\newtheorem{cor}[thm]{Corollary}

\newtheorem{lem}[thm]{Lemma}

\theoremstyle{definition}
\newtheorem{defn}[thm]{Definition}

\newtheorem{rmk}[thm]{Remark}
\newcommand{\sA}{\mathcal{A}}

\newcommand{\sE}{\mathcal{E}}
\newcommand{\sF}{\mathcal{F}}
\newcommand{\sG}{\mathcal{G}}
\newcommand{\sH}{\mathcal{H}}

\newcommand{\sK}{\mathcal{K}}

\newcommand{\sO}{\mathcal{O}}

\newcommand{\sU}{\mathcal{U}}

\newcommand{\sW}{\mathcal{W}}

\newcommand{\mC}{\mathbb{C}}
\newcommand{\mD}{\mathbb{D}}

\newcommand{\mH}{\mathbb{H}}
\newcommand{\mL}{\mathbb{L}}

\newcommand{\mP}{\mathbb{P}}

\newcommand{\mZ}{\mathbb{Z}}

\newcommand{\mW}{\mathbb{W}}
\newcommand{\Ima}{\mathrm{Im}\,}

\newcommand{\rank}{\mathrm{rank}\,}

\usepackage{color}

\numberwithin{equation}{section}

\newcommand{\beba}  {\begin{equation}\begin{array}{rcl}}

\newcommand{\eaee}  {\end{array}\end{equation}}

\makeatletter
\def\l@section{\@tocline{1}{0pt}{1pc}{}{}}
\def\l@subsection{\@tocline{2}{0pt}{1pc}{4.6em}{}}
\def\l@subsubsection{\@tocline{3}{0pt}{1pc}{7.6em}{}}
\renewcommand{\tocsection}[3]{%
  \indentlabel{\@ifnotempty{#2}{\makebox[2.3em][l]{%
    \ignorespaces#1 #2.\hfill}}}#3}
\renewcommand{\tocsubsection}[3]{%
  \indentlabel{\@ifnotempty{#2}{\hspace*{2.3em}\makebox[2.3em][l]{%
    \ignorespaces#1 #2.\hfill}}}#3}
\renewcommand{\tocsubsubsection}[3]{%
  \indentlabel{\@ifnotempty{#2}{\hspace*{4.6em}\makebox[3em][l]{%
    \ignorespaces#1 #2.\hfill}}}#3}
\makeatother

\title{Fujita decomposition and Massey product for fibered varieties}

\author{Luca Rizzi}
\address{Luca Rizzi\\Graduate School of Mathematical Sciences \\
the University of Tokyo\\
Tokyo, 153-8914 Japan\\
\texttt{rizzil@ms.u-tokyo.ac.jp}}

\author{Francesco Zucconi}
\address{Francesco Zucconi\\Department of Mathematics, Computer Science and Physics \\
Universit\`a di Udine\\
Udine, 33100\\ Italia
\texttt{Francesco.Zucconi@dimi.uniud.it}}

\begin{document}

\markboth{}{}
\begin{abstract} Let  $f\colon X\to B$ be a semistable fibration where $X$ is a smooth variety of dimension $n\geq2$ and $B$ is a smooth curve.
We give the structure theorem for the local system of the relative $1$-forms and of the relative top forms. This gives a neat interpretation of the second Fujita decomposition of $f_*\omega_{X/B}$.
 We apply our interpretation to show the existence, up to base change, of higher irrational pencils and on the finiteness of the associated monodromy representations under natural Castelnuovo-type hypothesis on local subsystems. Finally we give a criteria to have that $X$ is not of Albanese general type if $B=\mP^1$.
\end{abstract}

\maketitle
\section{Introduction}

Let $f\colon X\to B$ be a morphism with general smooth fiber between a smooth variety $X$ of dimension $n\geq2$ and a smooth curve $B$. We denote by $\omega_X$ the canonical sheaf of a variety $X$. The fiber $f^{-1}(b)$ over a point $b\in B$ will be usually denoted by $F$ or $F_b$ if we want to make the base point explicit. We mainly focus on \emph{semistable fibrations} that is we assume that all the fibers are reduced and normal crossing divisors.

A basic tool to study fibrations is the \emph{relative dualizing sheaf} $\omega_{X/B}=\omega_X\otimes f^*\omega_B^\vee$ and in particular its direct image $f_*\omega_{X/B}$ which is a vector bundle on $B$ of general fiber $H^0(F, \omega_F)$.
By a famous result of Fujita, see \cite{Fu} and \cite{Fu2}, $f_*\omega_{X/B}$ has two splittings classically known as \emph{first and second Fujita decomposition}. They are respectively 
\begin{equation}
f_*\omega_{X/B}\cong \sO_B^h\oplus \sE
\end{equation} where $\sE$ is a locally free nef sheaf on $B$ with $h^1(B,\sE\otimes \omega_B)=0$
and
\begin{equation}
f_*\omega_{X/B}\cong \sU\oplus \sA
\end{equation} where $\sU$ is a unitary flat vector bundle and $\sA$ ample.
Putting these two decompositions together we have the more general 
\begin{equation}
f_*\omega_{X/B}\cong \sO_B^h\oplus \sU'\oplus \sA
\end{equation} where $\sU\cong\sO_B^h\oplus \sU'$ is a decomposition of flat bundles and $\sU'$ has no global sections.

Recall that there is a one to one correspondence modulo isomorphism between 
\begin{enumerate}[{i)}]
	\item flat vector bundles on $B$
	\item local systems of $\mC$ vector spaces on $B$
	\item representations of the fundamental group $\pi_1(B,b)$.
\end{enumerate} see \cite[Proposition 9.11]{Vo1} and \cite[Corollary 3.10]{Vo2}.
Therefore, naturally associated to $\sU$ there are also a local system and a representation of the fundamental group and in this paper we will often use this correspondence.

The flat bundle $\sU$ is strictly related to the variation of Hodge structure of $X\to B$ in the following way. The restriction $f^0 \colon X^0 = f^{-1}(B^0) \to  B^0$ of $f$ to the locus of regular values $B^0$ is smooth
and defines a geometric VHS of weight $n-1$ on the local system $R^{n-1}f^0_*\mZ$ over $B^0$.
Denoting as usual $\sF^p\subset \sH^{n-1}=R^{n-1}f^0_*\mC_X\otimes \sO_{B^0}$ the Hodge filtration, we have
$$
\sF^{n-1}\cong f_*\omega_{X/B|B_0}
$$ and, denoting by $j\colon B^0\hookrightarrow B$ the inclusion, we have that the local system associated to $\sU^0=j^*\sU$ is a sublocal system of $R^{n-1}f^0_*\mC_X$ with fiber contained in $H^{n-1,0}=H^0(\omega_F)$. Note that while $R^{n-1}f_*\mC_X$ is not a local system and in general there is no way to extend $R^{n-1}f^0_*\mC_X\otimes\sO_{B^0}$ on the whole $B$, this can be done with $\sU^0$ and its trivial extension is exactly $\sU$. This comes from the fact that the hermitian form induced by the intersection form on the fibers forces the local monodromies around the points in $B\setminus B^0$ to be trivial, see for example \cite{CD1} where it is proved that the extension of $\sU^0$ can be done also if $f$ is not semistable. Actually, building on the seminal papers by Fujita, a lot of work has been devoted to solve the problem of the semi-ampleness of $f_*\omega_{X/B}$; see: \cite{CD1}, \cite{CD2}, \cite{CD3}, \cite{CK}.

In this paper the study of $\sU$ and the associated monodromy is done by the theory of \emph{Massey products}, previously known as \emph{adjoint forms}. In \cite{PT} this study is done for a fibred surface $f\colon S\to B$; we generalise and extend the results of \cite{PT}. We recall that Massey products have been introduced in \cite{CP} and \cite{PZ} and then applied in \cite{Ra}, \cite{PR}, \cite{CNP}, \cite{victor}, \cite{BGN}, \cite{RZ1},
\cite{RZ2}, \cite{RZ3} and recently \cite{CRZ}. 
The basic construction is as follows. Consider the fiber $F$ of our fibration $f$ and take $\eta_1,\dots,\eta_n$ 1-forms on $F$ in the kernel of the cup product $\cup \xi\colon H^0(\Omega^1_{F})\to H^0(\sO_F)$ where $\xi$ is the associated infinitesimal deformation $\xi\in H^1(T_F)$. Choosing $s_1,\dots,s_n\in H^0(\Omega^1_{X|F})$ liftings of the $\eta_i$'s we have a top form $\Omega\in H^{0}(\omega_{X|F})$ from the element $s_{1}\wedge\ldots\wedge s_{n}$. Since $\omega_{X|F}\cong \omega_{F}$ we 
actually obtain from $\Omega$ a top form of the canonical sheaf $\omega_{F}$. Such a form is the classical \emph{adjoint form} or \emph{Massey product}. We will denote it by $m_{\xi}(\eta_1,\ldots,\eta_n)$. A central definition is the definition of Massey triviality: we say that the sections $\eta_1,\dots,\eta_n$ are Massey trivial if their Massey product $m_{\xi}(\eta_1,\ldots,\eta_n)$ is a linear combination of the top forms $\eta_1\wedge\dots\wedge\widehat{\eta_i}\wedge\dots\wedge\eta_n$, $i=1\dots,n$. If furthermore the $\eta_1\wedge\dots\wedge\widehat{\eta_i}\wedge\dots\wedge\eta_n$ are linearly independent in $H^0(\omega_F)$ we say that the $\eta_i$ form a \emph{strict} subspace of $H^0(\Omega^1_F)$.
Massey triviality has often proved to be a useful tool for solving Torelli-type problems, here we will present a different implementation of the theory.

In Section \ref{sezioneaggiunta} we construct Massey products in families, using the following idea, see also \cite{PT}. Take the pushforward via $f$ of the exact sequence $$0\to f^*\omega_B\to \Omega^1_X\to \Omega^1_{X/B}\to 0$$ where $\Omega^1_{X/B}$ is the sheaf of relative differentials, see Section \ref{sez1} for details. Now  the connecting morphism
$$
\partial \colon f_*\Omega^1_{X/B}\to R^1f_*\sO_F\otimes\omega_B
$$ restricted on the general fiber is exactly $\cup\xi$, and we will denote its kernel by $\ker\partial:=K_\partial$. A local sections of $K_\partial$ can be considered as a family of liftable 1-forms on each fiber, therefore by choosing $n$ local sections of $K_\partial$ we can construct a family of Massey products which roughly corresponds to glueing together the $m_{\xi}(\eta_1,\ldots,\eta_n)$ for all the fibers over the considered open subset. The key point to this however is to show that all the sections of $K_\partial$ can be lifted to $f_*\Omega^1_X$; this is proved in Lemma \ref{split}. Of course the notion of Massey triviality and strictness also can be precisely extended in families, see Section \ref{sezioneaggiunta} for details. 
In this section we also highlight the key differences between Massey products of local or global sections of $K_\partial$, see Proposition \ref{lglob} and \ref{locglob}.

Now if $\dim X=2$ it is not difficult to see that the local system coming from the second Fujita decomposition is in fact contained in $K_\partial$ and a central idea in \cite{PT} is to construct Massey products starting from sections of this local system instead of just taking arbitrary sections of $K_\partial$. In the general case however the situation is more difficult because the local system of the Fujita decomposition consists of top forms on the fibers, whereas $K_\partial$ consists of 1-forms on the fibers, hence the first  is not contained in the second. Therefore in Sections \ref{sez1} and \ref{sez2} we construct a new local system $\mD$ which is contained in $K_\partial$. In Section \ref{sez2} in particular we show how $\mD$ is related to the sheaf of \textit{closed relative differential 1-forms} $\Omega^1_{X/B,d_{X/B}}$
and we also show how this interpretation applied on the sheaf of \textit{closed relative differential (n-1)-forms} $\Omega^{n-1}_{X/B,d_{X/B}}$ recovers exactly the local system of the Fujita decomposition, see Theorem \ref{fujitaii}. By analogy with $\mD$ we denote it by $\mD^{n-1}$ and we have $\mD^{n-1}\otimes \sO_B=\sU$.

Now let $A\subset B$ a contractible open subset and $W\subset \Gamma(A, \mD)$ a vector subspace of dimension at least $n$. We say that $W$ is \emph{Massey trivial} if any $n$-uple of sections in $W$ is Massey trivial, see Definition \ref{mastriv}.
Furthermore we say that a sublocal system $\mW<\mD$ is \emph{Massey trivial generated} if its general fiber is generated under the monodromy action of $\mD$ by a Massey trivial vector space $W$, see Definition \ref{mastrivgen}.
These objects are strictly related to the Castelnuovo-de Franchis theorem. In its classical formulation this result relates the existence of a non constant holomorphic map $S\to C$ from a surface to a curve of genus $g\geq2$  to the existence of two linearly independent 1-forms on $S$ with vanishing wedge product. 
Since then it has been generalized to higher dimensional varieties, see \cite{Ca2} and \cite{Ran}.

Recall that a \emph{higher irrational pencil} is a morphism with connected fibers $X\to Y$ with
target a normal variety $Y$ of maximal Albanese dimension and irregularity greater than its dimension.
Let $W$ be a Massey trivial subspace. Call $H$ the kernel of the monodromy representation of $\mD$, which is a normal subgroup of $\pi_1(B, b)$, and call $H_W$ the subgroup of $H$ which acts trivially on W. For every subgroup $K<H_W$, we denote by $B_K\to B$ the \'{e}tale base change of group $K$ and by $X_K\to B_K$ the associated fibration. In Section \ref{sez4} we prove the following theorem, which is a refinement of the generalized Castelnuovo-de Franchis \cite[Theorem 1.14]{Ca2}
\begin{thml}
\label{A}
Let $X\to B$ be a semistable fibration with $\dim X=n$. Let $A\subset B$ be an open subset and $W\subset \Gamma(A,\mD)$ a Massey trivial strict subspace.
Then $X_K$ has a higher irrational pencil $h_K\colon X_K\to Y$ over a normal $(n-1)$-dimensional variety of general type $Y$ such that $W\subset h_K^*(H^0(Y,\Omega^1_Y))$. Furthermore if $W$ is a maximal Massey trivial strict subspace we have the equality $W= h_K^*(H^0(Y,\Omega^1_Y))$.
\end{thml}

In the remaining of Section \ref{sez4} we study the monodromy of a Massey trivial generated local system $\mW$. Call $\rho_\mW $ the action of the fundamental group $\pi_1(B, b)$ on the stalk of $\mW$ and call $G_\mW=\pi_1(B, b)/\ker \rho_\mW$ the monodromy group. 
We construct an action of this group on a suitable set $\sK$ of morphisms from the fiber $F$ to $Y$ as in Theorem \ref{A} and thanks to this action we prove
\begin{thml}
\label{B}
Let $f \colon X \to B$ be a semistable fibration on a smooth projective curve B and let $\mW<\mD$ be a strict Massey trivial generated local system.
Then the associated monodromy group $G_\mW$ is finite and the fiber of $\mW$ is isomorphic to 
$$
\sum_{k\in \sK} k^*H^0(Y,\Omega^1_Y).
$$
\end{thml}
As a corollary we obtain the following result on the monodromy of $\mD$ and $\mD^{n-1}$
\begin{cor}
If $\mD$ is Massey trivial generated by a strict subspace, then its monodromy group is finite.
If furthermore the map $\bigwedge^{n-1}\mD\to \mD^{n-1}$ is surjective, the local system $\mD^{n-1}$ also has finite monodromy. 
\end{cor}
See Corollary \ref{hyper} for an example where this is applied.

Recall that the finiteness of the monodromy group of a local system is equivalent to the semi-ampleness of the unitary flat vector bundle, see for example \cite[Theorem 2.5]{CD1}, hence since $\mD^{n-1}$ is the local system associated to $\sU$, this corollary is indeed a result on the semi-ampleness of $\sU$ and hence of $f_*\omega_{X/B}$ since $f_*\omega_{X/B}=\sU\oplus\sA$ and $\sA$ is ample. We expect generalisation of our corollary in the light of \cite{CK}.

The rest of this paper is structured as follows.
In Section \ref{sectionclosed} we introduce a new local system on $B$ which contains all the Massey products obtained by
sections of $\mD$ but without the ambiguity given by the choice of the liftings and we prove that the vanishing of this local system is strictly related to Massey triviality of \emph{global sections} of $\mD$. 
In Section \ref{sez6} we highlight the relation between Massey triviality and the first Fujita decomposition.
Finally in Section \ref{sez7} we find a bound for the integer $h$ in the first Fujita decomposition $f_*\omega_{X/B}=\sO_B^h\oplus \sE$ and we focus on the case of a fibration on $\mP^1$. Note that in this case the first and second Fujita decompositions coincide, that is $f_*\omega_{X/B}=\sO_B^h\oplus \sA$. We denote by $r$ the rank of the subsheaf generically generated by the global sections of $f_*\omega_{X/B}$. We prove:
\begin{thml} Let  $X$ be an irregular variety with $q(X)>n$. If $f\colon X\to\mP^1$ is a fibration with $r=p_g(F)$ then $X$ is not of Albanese general type
\end{thml}

\section{Fibration on curves and a local system of 1-forms}
\label{sez1}

Let $f\colon X\to B$ be a {\it{semistable fibration}} of an $n$-dimensional smooth projective variety $X$ over a smooth projective curve $B$. Denote by $B_0$ the locus of singular values of $f$ and by $B^0=B\setminus B_0$ the open set of regular values.
\subsection{Torsion freeness of the sheaf of relative differentials}
The exact sequence 
\begin{equation}
\label{diffrel}
0\to f^*\omega_B\to \Omega^1_X\to \Omega^1_{X/B}\to 0 
\end{equation} defines the sheaf of relative differentials $\Omega^1_{X/B}$.

This sheaf is not torsion free for a general fibration $f$, but it turns out to be torsion free in our setting, when the fibers of $f$ are reduced and normal crossing divisors.
An easy way to see this is using the sheaf of logarithmic differential forms; see  \cite{De}.

Recall that if $X$ is a smooth variety and $D$ is a normal crossing divisor on $X$, locally given on an open set $U$ by $z_1z_2\cdots z_k=0$, we can define the sheaf $\Omega^1_X(\text{log }D)$ of \emph{logarithmic differentials} as the locally free $\sO_X$-module generated by $dz_1/z_1,\ldots,dz_k/z_k,dz_{k+1},\ldots,dz_{n}$.
This sheaf fits into the following exact sequence
\begin{equation}
0\to \Omega^1_X\to \Omega^1_X(\text{log }D) \stackrel{\text{Res}}{\rightarrow} \bigoplus_i \sO_{D_i}\to 0
\end{equation} where $D_i$ are the irreducible components of $D$ and $\text{Res}$ is the residue map.

In the case of a semistable fibration, it is not difficult to check, for example by local computation, that we have an injection 
\begin{equation}
0\to f^*\omega_B(\text{log }B_0)\to \Omega^1_X(\text{log }f^{-1}B_0)
\end{equation} with locally free cokernel which is denoted by $\Omega^1_{X/B}(\text{log})$, giving us the short exact sequence 
\begin{equation}
\label{logdiffrel}
0\to f^*\omega_B(\text{log }B_0)\to \Omega^1_X(\text{log }f^{-1}B_0)\to \Omega^1_{X/B}(\text{log})\to 0.
\end{equation} Note that this sequence is the locally free version of (\ref{diffrel}), more precisely call $P_i$ the points in $B_0$ and $E_j$ the irreducible components of $f^{-1}B_0$, that is the irreducible components of the singular fibers, we have the following diagram

\begin{equation}
\xymatrix{
&0\ar^{(1)}[d]&0\ar[d]&0\ar[d]&\\
0\ar[r]& f^*\omega_B\ar[d]\ar[r]&\Omega^1_X\ar[d]\ar[r]&\Omega^1_{X/B}\ar[r]\ar^{(3)}[d]&0\\
0\ar[r]& f^*\omega_B(\text{log }B_0)\ar[d]\ar[r]&\Omega^1_X(\text{log }f^{-1}B_0)\ar[d]\ar[r]&\Omega^1_{X/B}(\text{log})\ar[r]\ar^{(4)}[d]&0\\
0\ar[r]&\bigoplus_i f^*\mC_{P_i}\ar^{(2)}[r]\ar[d] &\bigoplus_j \sO_{E_j}\ar[r]\ar[d]&\bigoplus_j \sO_{E_j}/\bigoplus_i f^*\mC_{P_i}\ar[r]\ar[d]&0\\
&0&0&0&
}
\label{diagrammalog}
\end{equation}
The only arrows that need an explanation are the following
\begin{enumerate}
	\item This injectivity comes from the flatness of $f$.
	\item This arrows is the obvious one coming from the maps $f^*\mC_{P_i}\cong \sO_{f^{-1}P_{i}}\to \sO_{E_j}$ where $E_j\subset f^{-1}P_i$. It is injective for some $E_j\subset f^{-1}P_i$ because the fibers are reduced.
	\item This arrows, which exists by commutativity, is injective because (2) is 
	\item This arrow also follows by the commutativity of the diagram
\end{enumerate}

In particular we have that $\Omega^1_{X/B}$ is torsion free because by arrow (3) it is a subsheaf of a locally free sheaf.

Something similar can also be done in the case where the base of the fibration is not a curve, for details see \cite{Il} and \cite{MR}.

In the following we will denote by $\Omega^k_{X/B}$ the exterior powers of the sheaf of relative differentials, i.e.
\begin{equation}
\Omega^k_{X/B}=\bigwedge^k\Omega^1_{X/B}.
\end{equation}

The relative dualizing sheaf on the other hand is defined by
\begin{equation}
\omega_{X/B}=\Omega^n_X\otimes f^*\omega_B^\vee 
\end{equation} and it is locally free since both $X$ and $B$ are smooth.
Recall that $f_*\Omega^{n-1}_{X/B}$ and $f_*\omega_{X/B}$ are isomorphic when restricted to $B^0$. Call $f_*\omega_{X/B}$ {\it{the Fujita sheaf of the fibration $f\colon X\to B$}}.

\subsection{Sub-sheaves of the Fujita sheaf}
It turns out that $f_*\omega_{X/B}$ contains a couple of sub-sheaves which are interesting for our purposes, in the remaining of this sections we will introduce both of them.

Over each regular value $b\in B^0$, Sequence (\ref{diffrel}) restricted to the fiber $F_b$ is the exact sequence
\begin{equation}
\label{solita}
0\to T_{B,b}^\vee\otimes\sO_{F_b}\to \Omega^1_{X|F_b}\to \Omega^1_{F_b}\to 0 
\end{equation} which gives an infinitesimal deformation of the fiber $F_b$, $\xi_b\in \text{Ext}^1(\Omega^1_{F_b},\sO_{F_b})\cong H^1(F_b, T_{F_b})$.

Take the pushforward of sequence (\ref{diffrel})
\begin{equation}
0\to f_*f^*\omega_B\to f_*\Omega^1_X\to f_*\Omega^1_{X/B}\to R^1f_*f^*\omega_B\to \dots
\end{equation} which by projection formula is 
\begin{equation}
\label{seq1}
0\to \omega_B\to f_*\Omega^1_X\to f_*\Omega^1_{X/B}\to R^1f_*\sO_X \otimes\omega_B\to \dots
\end{equation}
Over each regular value $b\in B^0$, this sequence gives 
\begin{equation}
0\to H^0(\sO_{F_b})\to H^0(\Omega^1_{X|F_b})\to H^0(\Omega^1_{F_b})\stackrel{\delta_{\xi_b}}{\rightarrow} H^1(\sO_{F_b})\to \dots
\end{equation}
which is exactly the cohomology long exact sequence of (\ref{solita}). In particular recall that the connecting morphism $\delta_{\xi_{b}}$ is given by the cup product with the Kodaira-Spencer class $\xi_b$.  

\subsubsection{ $K_\partial$ or the sheaf of fiberwise liftable $1$-forms} We can now define the first of the two above mentioned sheaves.
\begin{defn}
The sheaf $K_\partial$ is the kernel of the map $\partial\colon f_*\Omega^1_{X/B}\to R^1f_*\sO_X \otimes\omega_B$ of sequence (\ref{seq1}).
\end{defn}
We note that, over the general $b\in B^0$, it holds that 
\begin{equation}
\label{kerxi}
K_\partial\otimes \mC(b)= \ker \delta_{\xi_b}
\end{equation} that is, over suitable subsets $A\subset B$, we can think of $K_\partial$ as the sheaf of holomorphic one forms on the fibers of $f$ which are liftable to the variety $X$.

The sheaf $K_\partial$ is not locally free in general, but it is in the case of a semistable fibration. In fact recall that the sheaf of relative differentials $\Omega^1_{X/B}$ is torsion free, hence its direct image $f_*\Omega^1_{X/B}$ is torsion free and hence locally free on the curve $B$. It follows that $K_{\partial}$ is torsion free, hence again locally free on $B$.

\begin{lem}
\label{split}
If $X\to B$ is a semistable fibration, the exact sequence 
\begin{equation}
0\to \omega_B\to f_*\Omega^1_X\to K_\partial\to 0
\label{seqK}
\end{equation} splits.
\end{lem}
\begin{proof}
The proof is the same as \cite[Lemma 3.5]{PT} and it is standard once we made sure that $K_\partial$ is locally free. The central idea is that the splitting of this sequence is equivalent to the splitting of its dual tensored by $\omega_B$, that is 
$$
0\to K_\partial^\vee\otimes\omega_B\to {f_*\Omega^1_X}^\vee\otimes \omega_B\to \sO_B\to 0
$$
This splitting is equivalent to the vanishing of the cohomology map $H^0(\sO_B)\to H^1(K_\partial^\vee\otimes\omega_B)$ which by duality is equivalent to the vanishing of $H^0(K_\partial)\to H^1(\omega_B)$. By Sequence (\ref{seqK}) we prove instead the injectivity of $H^1(\omega_B)\to H^1(f_*\Omega^1_X)$ and this comes from the fact that if we compose this map with the injective map given by the Leray spectral sequence we obtain
$$
H^1(B,\omega_B)\to H^1(B,f_*\Omega^1_X)\to H^1(X,\Omega^1_X)
$$ which is still injective because  $H^1(B,\omega_B)\to H^1(X,\Omega^1_X)$ sends the class of a point to the class of a fiber.
\end{proof}

\begin{rmk}
\label{remark}
Call $f^0\colon X^0\to B^0$ our fibration restricted to the locus of regular values. Denote as usual by $\sH^{p,q}$ the Hodge bundles on $B^0$. The map $\partial\colon f_*\Omega^1_{X/B}\to R^1f_*\sO_{X} \otimes\omega_{B}$ on $B^0$ is exactly the variation of Hodge structure
\begin{equation}
\overline{\nabla}^{1,0}\colon \sH^{1,0}\to \sH^{0,1}\otimes \omega_B
\label{variation}
\end{equation}
\end{rmk}
See \cite[Section 10.2]{Vo1}.
\subsubsection{ $\mD$ or the sheaf of fiberwise forms liftable to $d$-closed $1$-forms}
The second sheaf we are interested in, which will turn out to be a local system on $B$, comes from considering the de Rham closed holomorphic differential forms.

Consider the de Rham sequences on $X$ and on $B$ respectively:
\begin{equation}
0\to \mC_X\to \sO_X\to \Omega^1_{X}\to \Omega^2_{X}\to \dots\to \Omega^n_{X}\to 0
\label{DRX}
\end{equation} and 
\begin{equation}
0\to \mC_B\to \sO_B\to \omega_B\to 0
\label{DRB}
\end{equation}
From (\ref{DRX}) we get the short exact sequence 
\begin{equation}
0\to \mC_X\to \sO_X\to \Omega^1_{X,d}\to 0
\label{1chiuse}
\end{equation} where $\Omega^1_{X,d}$ is the sheaf of de Rham closed holomorphic differential forms. 

Comparing the pushforward of (\ref{1chiuse}) together with (\ref{DRB}) we get the following commutative diagram
\begin{equation}
\xymatrix{
0\ar[r]& f_*\mC_X\ar @{=}[d]\ar[r]&f_*\sO_X\ar @{=}[d]\ar[r]&f_*\Omega^1_{X,d}\ar[r]&R^1f_*\mC_X\ar[r]&R^1f_*\sO_X\ar[r]&\dots\\
0\ar[r]&\mC_B\ar[r]&\sO_B\ar[r]&\omega_B\ar @{^{(}->}[u]\ar[r]&0
}
\label{dia1}
\end{equation}

\begin{defn}
The sheaf $D$ is defined as the cokernel of the vertical map $\omega_B\to f_*\Omega^1_{X,d}$. Alternatively by diagram (\ref{dia1}) it is the kernel of the map $R^1f_*\mC_X\to R^1f_*\sO_X$.
\end{defn}
 
\begin{lem}
\label{lemmainclusione}
We have an inclusion of sheaves $D\hookrightarrow K_{\partial}$
\end{lem}
\begin{proof}
Immediate by the fact that the exact sequence defining $D$
\begin{equation}
0\to \omega_B\to f_*\Omega^1_{X,d}\to D\to 0
\end{equation} is basically sequence (\ref{seqK}) restricted to the closed holomorphic one-forms.
\end{proof}
We can therefore interpret $D$ as the sheaf of  holomorphic one-forms on the fibers of $f$ which are liftable to \textit{closed} holomorphic forms of the variety $X$.

\begin{lem}
\label{loc}
The sheaf $D$ is a local system.
\end{lem}
\begin{proof}
See \cite[Lemma 2.4]{PT}. The idea of the proof is the following. Call $j\colon B^0 \to B$ the injection on the locus corresponding to the smooth fibers of $f$. Since $D$ is a subsheaf of $R^1f_*\mathbb C$, its restriction $j^*D$ is a subsheaf of the local system $j^*R^1f_*\mathbb C$, therefore $j^*D$ is itself a local system. Now it is standard to prove that $j^*D$ extends to a local system on the whole curve $B$. To do this it is enough to prove that the monodromy around the singular values is trivial. Since by definition $D$ is the kernel of the morphism $R^1f_*\mC_X\to R^1f_*\sO_X$, its stalk over a point $b\in B^0$ is contained in the kernel of the projection map $H^1(F_b,\mC)\to H^{0,1}(F_b)$ and so it is a vector subspace of $H^{1,0}(F_b)$. The standard Hermitian form is positive definite on $H^{1,0}(F_b)$, hence the monodromy representation associated to $j^*D$ is unitary flat, furthermore it is unipotent by the Monodromy Theorem (see \cite[Theorem 11.8]{PS}), hence it is trivial. This proves that $j_*j^*D$ is a local system.

The last step consists in noticing that $  R^1f_*\mC_X\to j_*j^*R^1f_*\mC_X$ is surjective by the local invariant monodromy theorem, see for example \cite[Theorem 5.3.4]{CEZGT}, so $D\to j_*j^*D$ is also surjective.
Since $D\to j_*j^*D$ is an isomorphism on $B^0$, it immediately follows that $D\to j_*j^*D$ is an isomorphism because otherwise the kernel would be a torsion subsheaf of $D$, and hence trivial since $D\hookrightarrow K_\partial$ and $K_\partial$ is locally free.
%
\end{proof}

From now on we will denote by $\mD$ the local system $D$.

\begin{rmk}We also have that $j^*\mD$ is the largest local subsystem of $j^*R^1f_*\mC$ with stalk a subspace of $H^{1,0}(F_b)$ on the general fiber. Indeed, every other subsystem with the same property is contained in the kernel of $R^1f_*\mC_X\to R^1f_*\sO_X$ and so it is contained in $\mD$.
\end{rmk}

When $X$ is a surface, this means that $\mD$ is the local system which gives the second Fujita decomposition of $f_*\omega_{X/B}$. 
In the next section we present how to generalize this result when $\dim X>2$.

\section{Relation between $\mD$ and the sheaf of relatively closed differential forms}
We want to find the analogue of $\mD$ in the case of volume forms. For this we present an alternative description of $\mD$.
\label{sez2}
%
%
Consider the sequence (\ref{diffrel})
\begin{equation*}
0\to f^*\omega_B\to \Omega^1_X\to \Omega^1_{X/B}\to 0 
\end{equation*} and its wedge 
\begin{equation}
0\to f^*\omega_B\otimes \Omega^1_{X/B}\to \Omega^2_X\to \Omega^2_{X/B}\to 0. 
\end{equation}
Note that the map $f^*\omega_B\otimes \Omega^1_{X/B}\to \Omega^2_X$ is not injective in general and its kernel is a torsion sheaf. In the semistable case however, since we have seen that $\Omega^1_{X/B}$ is torsion free, the kernel must be trivial because $f^*\omega_B\otimes \Omega^1_{X/B}$ is torsion free.

More in general we have the exact sequences 
\begin{equation}
\label{seqwedge}
0\to f^*\omega_B\otimes \Omega^k_{X/B}\to \Omega^{k+1}_X\to \Omega^{k+1}_{X/B}\to 0.
\end{equation}
At the level of top forms on $X$ however, one does not  have the isomorphism
\begin{equation}
\label{noniso}
f^*\omega_B\otimes \Omega^{n-1}_{X/B}\cong \Omega^n_X
\end{equation} which holds only in the case of a smooth submersion.
On the other hand we have the exact sequence 
\begin{equation}
0\to f^*\omega_B\otimes \Omega^{n-1}_{X/B}\to  \Omega^n_X\to \Omega^n_{X{|Z}}\to 0
\end{equation}
where $Z$ is the subscheme of $X$ given by the critical points, and this is enough to deduce that we still have the injection
\begin{equation}
\label{inclusione}
f_*\Omega^{n-1}_{X/B}\hookrightarrow f_*\omega_{X/B}.
\end{equation}

\subsection{$\mD$ as a subsheaf of the sheaf of relatively closed differential forms}

We introduce now the sheaf of relatively closed differential forms. The de Rham differential $d$ on $\Omega^1_X$ induces a differential $d_{X/B}$ on the sheaf of relative differential forms.
In fact we have a commutative diagram 
\begin{equation}
\xymatrix{
0\ar[r]&f^*\omega_B\ar[r]\ar^{d}[d]&\Omega^1_X\ar[r]\ar^{d}[d]&\Omega^1_{X/B}\ar[r]&0\\
0\ar[r]&f^*\omega_B\otimes \Omega^1_{X/B}\ar[r]&\Omega^2_X\ar[r]&\Omega^2_{X/B}\ar[r]&0
}
\end{equation}
which can be completed 
\begin{equation}
\label{diagrammadeRham}
\xymatrix{
0\ar[r]&f^*\omega_B\ar[r]\ar^{d}[d]&\Omega^1_X\ar[r]\ar^{d}[d]&\Omega^1_{X/B}\ar[r]\ar^{d_{X/B}}[d]&0\\
0\ar[r]&f^*\omega_B\otimes \Omega^1_{X/B}\ar[r]&\Omega^2_X\ar[r]&\Omega^2_{X/B}\ar[r]&0
}
\end{equation}

\begin{defn}
The kernel of $d_{X/B}$ is by definition the sheaf of closed relative differential forms, denoted by $\Omega^1_{X/B,d_{X/B}}$.
\end{defn}

Diagram (\ref{diagrammadeRham}) can be completed with the exact sequence of the kernels

\begin{equation}
\label{diagrammadeRhamker}
\xymatrix{
0\ar[r]&C\ar[r]\ar@{^{(}->}[d]&\Omega^1_{X,d}\ar[r]\ar@{^{(}->}[d]&\Omega^1_{X/B,{d_{X/B}}}\ar[r]\ar@{^{(}->}[d]&0\\
0\ar[r]&f^*\omega_B\ar[r]\ar^{d}[d]&\Omega^1_X\ar[r]\ar^{d}[d]&\Omega^1_{X/B}\ar[r]\ar^{d_{X/B}}[d]&0\\
0\ar[r]&f^*\omega_B\otimes \Omega^1_{X/B}\ar[r]&\Omega^2_X\ar[r]&\Omega^2_{X/B}\ar[r]&0
}
\end{equation}

%
%

It is not difficult to realize that $C$ consists of holomorphic forms which can be locally expressed as $f^*(a(t)dt)$ with $a(t)$ an holomorphic function on $B$, therefore pushing forward the kernel sequence of Diagram (\ref{diagrammadeRhamker}) we obtain
\begin{equation}
\label{seq2}
0\to \omega_B\to f_*\Omega^1_{X,d}\to f_*\Omega^1_{X/B,d_{X/B}}\to R^1f_*C\to\dots
\end{equation}

\begin{prop}
The local system $\mD$ introduced in the previous section is the kernel of the map $f_*\Omega^1_{X/B,d_{X/B}}\to R^1f_*C$. In particular $\mD$ is the subsheaf of $f_*\Omega^1_{X/B,d_{X/B}}$ whose sections locally come from closed holomorphic 1-forms on $X$.
\end{prop}
\begin{proof}
Recall that $\mD$ by definition is the cokernel of $\omega_B\to f_*\Omega^1_{X,d}$, hence by Sequence (\ref{seq2}) it can be also seen as the kernel of $f_*\Omega^1_{X/B,d_{X/B}}\to R^1f_*C$.
\end{proof}

The reason why we give this equivalent definition of $\mD$ is because it will allow us to extend the constructions presented in the previous section to the case of top differential forms.

To visualize the relation among the sheaves $K_{\partial}$, $\mD$ and $f_*\Omega^1_{X/B,d_{X/B}}$, we have the diagram
\begin{center}
\begin{equation}
\xymatrix{
&K_{\partial}\ar@{^{(}->}[dr]&\\
\mD\ar@{^{(}->}[ur]\ar@{^{(}->}[dr]&&f_*\Omega^1_{X/B}\\
&f_*\Omega^1_{X/B,d_{X/B}}\ar@{^{(}->}[ur]&
}
\end{equation}
\end{center}

%

\subsection{$D^{n-1}$ or the sheaf of fiberwise volume forms liftable to closed holomorphic forms}

Now take the analogous of Diagram (\ref{diagrammadeRhamker}) for $(n-1)$-forms

\begin{equation}
\label{diagrammadeRhamkertop}
\xymatrix{
0\ar[r]&C^{n-1}\ar[r]\ar@{^{(}->}[d]&\Omega^{n-1}_{X,d}\ar[r]\ar@{^{(}->}[d]&\Omega^{n-1}_{X/B,{d_{X/B}}}\ar[r]\ar@{^{(}->}[d]&0\\
0\ar[r]&f^*\omega_B\otimes\Omega^{n-2}_{X/B} \ar[r]\ar^{d}[d]&\Omega^{n-1}_X\ar[r]\ar^{d}[d]&\Omega^{n-1}_{X/B}\ar[r]\ar^{d_{X/B}}[d]&0\\
0\ar[r]&f^*\omega_B\otimes \Omega^{n-1}_{X/B}\ar[r]&\Omega^n_X\ar[r]&\Omega^n_{X}|_Z\ar[r]&0.
}
\end{equation}
and take the pushforward of the kernel sequence.

As in the previous case a local computation shows that $f_*C^{n-1}=\omega_B\otimes f_*\Omega_{X/B}^{n-2}$ and we have the long exact sequence
\begin{equation}
0\to \omega_B\otimes f_*\Omega_{X/B}^{n-2}\to f_*\Omega_{X,d}^{n-1}\to f_*\Omega^{n-1}_{X/B,d_{X/B}}\to R^1f_*C^{n-1}\to\dots
\end{equation}
\begin{defn}
We denote by $D^{n-1}$ the coker of $\omega_B\otimes f_*\Omega_{X/B}^{n-2}\to f_*\Omega_{X,d}^{n-1}$. Alternatively it is the kernel of $f_*\Omega^{n-1}_{X/B,d_{X/B}}\to R^1f_*C^{n-1}$.
\end{defn}

Similarly to $\mD$, $D^{n-1}$ can be interpreted as a sheaf of forms on the fibers liftable to closed holomorphic forms on  $X$.

We will show that $D^{n-1}$ is a local system, but first note that we have the map of long exact sequences 
\begin{equation}
\label{diag}
\xymatrix{
0\ar[r] &\omega_B\otimes f_*\Omega_{X/B}^{n-2}\ar[r]\ar@{=}[d]& f_*\Omega_{X,d}^{n-1}\ar[r]\ar@{^{(}->}[d]& f_*\Omega^{n-1}_{X/B,d_{X/B}}\ar[r]\ar@{^{(}->}[d]& R^1f_*C^{n-1}\ar[r]\ar[d]&\\
0\ar[r] &\omega_B\otimes f_*\Omega_{X/B}^{n-2}\ar[r]& f_*\Omega_{X}^{n-1}\ar[r]& f_*\Omega^{n-1}_{X/B}\ar[r]& \omega_B\otimes R^1f_*\Omega_{X/B}^{n-2}\ar[r]&\\
}
\end{equation}
On $B^0$ the inclusion $f_*\Omega^{n-1}_{X/B,d_{X/B}}\hookrightarrow f_*\Omega^{n-1}_{X/B}$ is an isomorphism and the arrow $ f_*\Omega^{n-1}_{X/B}\to \omega_B\otimes R^1f_*\Omega_{X/B}^{n-2}$ is the variation of Hodge structure
\begin{equation}
\overline{\nabla}^{n-1,0}\colon \sH^{n-1,0}\to \sH^{n-2,1}\otimes \omega_B,
\label{variationn}
\end{equation} see also Remark (\ref{remark}).

Denoting as usual $\sF^p\subset \sH^{n-1}=j^*R^{n-1}f_*\mC_X\otimes \sO_{B^0}$ the Hodge filtration, the Diagram \cite[Page 251]{Vo1} becomes
\begin{equation}
\label{hodge}
\xymatrix{
\nabla\colon&0\ar[r]\ar[d]&\sF^{n-1}\otimes \omega_B\ar[d]\\
\nabla\colon&\sF^{n-1}\ar[r]\ar@{=}[d]&\sF^{n-2}\otimes \omega_B\ar[d]\\
\overline{\nabla}^{n-1,0}\colon&\sH^{n-1,0}\ar[r]& \sH^{n-2,1}\otimes \omega_B
}
\end{equation}
\begin{lem}
\label{lem1}
The sheaf $j^*D^{n-1}$ is the largest local system on $B^0$ contained in $f_*\Omega^{n-1}_{X/B}$.
\end{lem}
\begin{proof}
For better clarity we denote by $\nabla$ the flat connection $\nabla\colon\sH^{n-1}\to \sH^{n-1}\otimes \omega_B$ and by $\nabla|_{\sF^{n-1}}$ its restriction on $\sF^{n-1}$, that is the middle map of Diagram (\ref{hodge}).

We will prove our thesis showing that $D^{n-1}$ is the kernel of $\nabla|_{\sF^{n-1}}$. In fact 
$$
\ker \nabla|_{\sF^{n-1}}\subset \ker \nabla=R^{n-1}f_*\mC_X
$$ 
therefore $\ker \nabla|_{\sF^{n-1}}$ is a local system contained in $\sF^{n-1}$. It is also the largest with this property because any other local system contained in $\sF^{n-1}$ must also be contained in $R^{n-1}f_*\mC_X$ and hence it must be contained in $\ker \nabla|_{\sF^{n-1}}$.

Recall also the Cartan-Lie formula which allows to compute the connection $\nabla$; see \cite[Section 9.2.2]{Vo1}.
Take a suitable open subset $U\subset B^0$. If $\Omega$ is a differential form of degree $n-1$ on $f^{-1}(U)$ such that its restriction to $F_b$ is closed for all $b\in U$, then the map $b \to [\Omega|_{F_b}] \in H^{n-1}(F_b,\mC)$ is a section of the bundle $\sH^{n-1}$ which we shall denote by $\omega$.
Under these assumptions, if $u\in T_{B,t}$ is a tangent vector and $v\in \Gamma(T_{X|_{X_t}})$ is such that $f_*v=u$, then
\begin{equation}
\nabla_u(\omega)(t)=[\text{int}_v (d\Omega|_{X_t})].
\label{cartan}
\end{equation}

Now take a section $\omega$ of $D^{n-1}$. With a suitable choice of $U$, we have that $\Omega$ is closed since $D^{n-1}$ is the image of $f_*\Omega_{X,d}^{n-1}\to f_*\Omega^{n-1}_{X/B,d_{X/B}}$.
By (\ref{cartan}) it follows that $\nabla(\omega)=0$ since $d\Omega=0$, that is $D^{n-1}$ is contained in $\ker \nabla|_{\sF^{n-1}}$. 

Viceversa if $\omega$ is in the kernel of $\nabla|_{\sF^{n-1}}$, it is also in the kernel of $\overline{\nabla}^{n-1,0}$ by Diagram (\ref{hodge}). This kernel is the image of $f_*\Omega_{X}^{n-1}\to f_*\Omega^{n-1}_{X/B}$ by the second row of (\ref{diag}), that is we can choose $\Omega$ to be a holomorphic form. In particular $d\Omega$ is a holomorphic $n$-form on $f^{-1}(U)$. By the hypothesis $\omega\in \ker  \nabla|_{\sF^{n-1}}$ and by the Cartan-Lie formula we have that for every $t\in U$ and for every vector $u\in T_{B,t}$
$$
[\text{int}_v (d\Omega|_{X_t})]=0
$$ from which we deduce that $d\Omega|_{X_t}=0\in H^0(X_t, \Omega^{n}_{X|_{X_t}})$. Since this holds for all $t$ in $U$, $d\Omega$ must be zero, i.e. $\Omega$ is closed and we have proved that $\omega$ is in the image of $f_*\Omega^{n-1}_{X,d}$ which is $D^{n-1}$.

\end{proof}

\begin{rmk}
We have seen that the Cartan-Lie formula basically tells us that the kernel of $\nabla|_{\sF^{n-1}}$ is given by the section that are obtained as restriction on the fibers of a closed holomorphic form on $X$.

In \cite{Vo1}, also the map $\overline{\nabla}^{n-1,0}$ is explicitly computed using the Cartan-Lie formula and it turns out to be given by
$$
[\text{int}_v (\bar{\partial}\Omega|_{X_t})]
$$ that is the sections in $\ker \overline{\nabla}^{n-1,0}$ are obtained as restriction on the fibers of a holomorphic form on $X$, which is also evident by the second row of Diagram (\ref{diag}).
\end{rmk}

\begin{lem}
\label{lem2}
The local system $j^*D^{n-1}$ trivially extends on $B$, hence $D^{n-1}$ is a local system on $B$ which we will denote by $\mD^{n-1}$.
\end{lem}
\begin{proof}
The proof is essentially the same as for $\mD$.
The intersection form on $j^*D^{n-1}$ is, up to constant, strictly positive definite,  hence the
monodromy representation associated to $j^*D^{n-1}$ is unitary flat and unipotent by the
Monodromy Theorem, hence it is trivial. This proves that $j_*j^*D^{n-1}$ is a local system.
Exactly as we have seen before, $ R^{n-1}f_*\mC\to j_*j^*R^{n-1}f_*\mC$ is surjective and we easily deduce the isomorphism $j_*j^*D^{n-1}\cong D^{n-1}$.
\end{proof}

By Lemma (\ref{lem1}) and Lemma (\ref{lem2}) we get that $\mD^{n-1}$ is the largest local system contained in $f_*\Omega^{n-1}_{X/B}$. Recall that we have the inclusion  $f_*\Omega^{n-1}_{X/B}\hookrightarrow \omega_{X/B}$ (see (\ref{inclusione})) which is an isomorphism on $B^0$. It immediately follows that 
\begin{thm}
\label{secfujita}
$\mD^{n-1}$ is the local system that gives the second Fujita decomposition of $\omega_{X/B}$, that is 
\begin{equation}
\label{fujitaii}
\omega_{X/B}=\sU\oplus \sA
\end{equation} with $\sA$ ample and $\sU=\mD^{n-1}\otimes \sO_B$.
\end{thm}

\section{Massey products of sections of $\mD$}
\label{sezioneaggiunta}
We stress that the dimension of the fiber is $n-1$ and we now construct a map that allows us to define the \emph{Massey product or adjoint form} of $n$ sections of $K_\partial$.

\subsection{Definition of Massey product of $1$-forms}
Consider the map 
\begin{equation}
\bigwedge^nf_*\Omega^1_X\to f_*\bigwedge^n\Omega^1_X=f_*\omega_X
\end{equation}
and take the tensor product with $\omega_B^\vee=T_B$
\begin{equation}
\bigwedge^nf_*\Omega^1_X\otimes T_B\to f_*\omega_X\otimes T_B=f_*\omega_{X/B}. 
\end{equation}
Since sequence (\ref{seqK}) splits by Lemma (\ref{split}), the following wedge sequence also splits
\begin{equation}
\label{split2}
\xymatrix{
0\ar[r]& \bigwedge^{n-1} K_\partial\otimes \omega_B\ar[r]&\bigwedge^n f_*\Omega^1_X\ar[r]& \bigwedge^n K_{\partial}\ar[r]\ar @/_1.6pc/ [l]_{} & 0
}
\end{equation}
and it allows to define the composite map
\begin{equation}
\label{agg}
\bigwedge^n K_{\partial}\otimes T_B\to \bigwedge^n f_*\Omega^1_X\otimes T_B\to f_*\omega_{X/B} 
\end{equation}

Given $n$ sections $\eta_1,\ldots,\eta_n$ in $\Gamma(A,K_{\partial})$ on a suitable open subset $A\subset B$ consider the wedge products $\eta_1\wedge\dots\wedge\widehat{\eta_i}\wedge \dots\wedge\eta_n$ for $i=1,\ldots,n$ and their image via the map 
\begin{equation}
\label{wedge}
\bigwedge^{n-1} K_{\partial}\to \bigwedge^{n-1} f_*\Omega^1_{X/B}\to f_*\Omega^{n-1}_{X/B}\hookrightarrow f_*\omega_{X/B}.
\end{equation}

\begin{defn}
\label{omegai}
We call $\omega_i$, $i=1,\ldots,n$, the image of $\eta_1\wedge\dots\wedge\widehat{\eta_i}\wedge \dots\wedge\eta_n$ via map (\ref{wedge}) and $\sW$ the $\sO_B$-submodule of $\omega_{X/B}$ generated by the $\omega_i$.
\end{defn}
\begin{defn}
\label{mtrivial}
The \emph{Massey product or adjoint image} of $\eta_1,\ldots,\eta_n$ is the section of $f_*\omega_{X/B}$ computed via the morphism (\ref{agg}) starting with $\eta_1\wedge...\wedge\eta_n$. We say that the sections $\eta_1,\ldots,\eta_n$  are \emph{Massey trivial} if their Massey product is contained in the $\sO_B$-submodule $\sW$.
\end{defn}

Alternatively we can define Massey products pointwise: given a regular value $b\in B$, we can take the 1-forms defined by $\eta_1,\dots,\eta_n$ on the fiber $F_b$ and build their adjoint form in the usual way, which we briefly recall. For an extensive discussion on the topic, we refer to \cite{RZ1}.

Call $W<H^0(F_b,\Omega^1_{F_b})$ the $n$-dimensional subspace generated by these sections and take $s_1,\dots,s_n\in H^0(F_b,\Omega^1_{X|F_b})$ a choice of liftings on $X$. These lifting can be found since $W\subset \ker \delta_{\xi_b}$ by (\ref{kerxi}).

We have a top form $\Omega\in H^{0}(F_b,\omega_{X|F_b})$ from the 
element $s_{1}\wedge\ldots\wedge s_{n}\in 
\bigwedge^{n+1}H^{0}(F_b,\Omega^1_{X|F_b})$. Since $\omega_{X|F_b}\cong \omega_{F_b}$ we 
actually obtain from $\Omega$ a top form $\omega$ of $\omega_{F_b}$. Such an $\omega\in 
H^{0}(X,\omega_{F_b})$ is the classical \emph{adjoint form} or \emph{Massey product}. When needed we will denote it by $m_{\xi_b}(\eta_1,\ldots,\eta_n)$.

 All the pointwise defined adjoint forms can be glued together to an element of $f_*\omega_{X/B}$ modulo the $\sO_B$-submodule $\sW$. This construction agrees with the previous one of Definition \ref{mtrivial} on suitable open subsets $A\subset B$.


\subsubsection{$\mD$ and Massey product}
Of course since $\mD$ is a subsheaf of $K_{\partial}$, it makes sense to construct Massey products starting from sections of $\mD$, i.e. consider the map 
\begin{equation}
\bigwedge^n \mD\otimes T_B\to f_*\omega_{X/B}
\end{equation} instead of (\ref{agg}).

On the other hand note that the map (\ref{wedge}) restricted on $\mD$ has image in $\mD^{n-1}$, that is
\begin{equation}
\bigwedge^{n-1}\mD\to \mD^{n-1}
\end{equation} because the wedge product of forms which can be lifted to closed holomorphic forms can also be lifted to a closed holomorphic form.

\subsection{Massey triviality and strictness}
Let $A\subset B$ a contractible open subset and $W\subset \Gamma(A, K_{\partial})$ a vector subspace of dimension at least $n$.

We give the following definition
\begin{defn}
\label{mastriv}
We say that $W$ is Massey trivial if any $n$-uple of sections in $W$ is Massey trivial (see Definition \ref{mtrivial}).
\end{defn}

\begin{defn}
\label{strict}
We say that $W$ is strict if the map 
$$
\bigwedge^{n-1}W\otimes \sO_A\to f_*\omega_{X/B_{|A}}
$$ is an injection of vector bundles.
\end{defn}

\begin{prop}
\label{wedge0}
Let $A\subset B$ be a contractible open subset and $W\subset \Gamma(A, K_{\partial})$ a Massey trivial strict subspace. Then there exists a unique $\widetilde{W}\subset \Gamma(A,f_*\Omega^1_S)$ lifting $W$ such that the map
$$
\psi\colon \bigwedge^{n}\widetilde{W}\otimes T_A\to f_*\omega_{X/B_{|A}}
$$ is zero.
\end{prop}
\begin{proof}
We work by induction on the dimension of $W$. We start with $\dim W=n$.

If $\dim W=n$ take $\eta_1,\ldots,\eta_{n}$ a basis of $W$ and $s_1,\dots,s_{n}\in \Gamma(A,f_*\Omega^1_S)$ arbitrary liftings via Lemma (\ref{split}). 

By hypothesis there exist $a_i$ holomorphic functions on $A$ such that
\begin{equation}
\label{ipotesi}
\psi(s_1\wedge\dots\wedge s_{n}\otimes \frac{\partial}{\partial t})=\sum^{n}_{i=1} a_i\omega_i
\end{equation} where the $\omega_i$ are as in Definition (\ref{omegai}).

Defining a new lifting for the element $\eta_i$:
\begin{equation*}
t_i:=s_i+(-1)^{n-i}a_i\cdot dt
\end{equation*} we have 
\begin{equation*}
\begin{split}
\psi(t_1\wedge\dots\wedge t_{n}\otimes \frac{\partial}{\partial t})=\psi(s_1\wedge\dots\wedge s_{n}\otimes \frac{\partial}{\partial t})-\sum_{i=1}^{n} a_i\psi(s_1\wedge\dots\wedge\widehat{s_i}\wedge\dots\wedge s_{n}\wedge dt\otimes \frac{\partial}{\partial t})=\\
\psi(s_1\wedge\dots\wedge s_{n}\otimes \frac{\partial}{\partial t})-\sum^{n}_{i=1} a_i\omega_i=0.
\end{split}
\end{equation*}

No we perform the induction step. Assume now that $\dim W=k>n$.

We take a basis $\eta_1,\dots,\eta_k$ of $W$ and by induction we can choose liftings $t_1,\dots,t_{k-1}$ of $\eta_1,\dots,\eta_{k-1}$ such that the map $\psi$ is zero on $\langle t_1,\dots,t_{k-1} \rangle$. Finally we take $s_k$ a lifting of $\eta_k$.

Calling $\phi$ the map in (\ref{wedge}), by the hypothesis of Massey triviality for every $n$ elements in $W$, there exists constants  $b,b'$ and $a_i, c_i$, $i=1,...,n-1$, such hat 
\begin{equation*}
\psi(t_1\wedge\dots\wedge t_{n-1}\wedge s_k\otimes\frac{\partial}{\partial t})=\sum^{n-1}_{i=1}a_i\phi(\eta_1\wedge\dots\wedge\widehat{\eta_i}\wedge\dots\wedge \eta_{n-1}\wedge\eta_k)+b \phi(\eta_1\wedge\dots\wedge \eta_{n-1})
\end{equation*} and 
\begin{equation*}
\begin{split}
\psi(t_1\wedge\dots\wedge t_{n-1}\wedge (\sum_{i=n}^{k-1}t_i+s_k)\otimes\frac{\partial}{\partial t})=\sum^{n-1}_{i=1}c_i\phi(\eta_1\wedge\dots\wedge\widehat{\eta_i}\wedge\dots\wedge \eta_{n-1}\wedge (\sum_{i=n}^{k-1}\eta_i+\eta_k))+\\+b' \phi(\eta_1\wedge\dots\wedge \eta_{n-1}).
\end{split}
\end{equation*}
Now by the induction hypotesis we have that 
\begin{equation*}
\psi(t_1\wedge\dots\wedge t_{n-1}\wedge s_k\otimes\frac{\partial}{\partial t})=\psi(t_1\wedge\dots\wedge t_{n-1}\wedge (\sum_{i=n}^{k-1}t_i+s_k)\otimes\frac{\partial}{\partial t})
\end{equation*}
and by the strictness of $W$ we can compare their expressions and obtain $b=b'$ and $a_i=c_i=0$.

Therefore 
\begin{equation*}
\psi(t_1\wedge\dots\wedge t_{n-1}\wedge s_k\otimes\frac{\partial}{\partial t})=b \phi(\eta_1\wedge\dots\wedge \eta_{n-1})
\end{equation*} and choosing $t_k=s_k-bdt$ we have that 
\begin{equation*}
\psi(t_1\wedge\dots\wedge t_{n-1}\wedge t_k\otimes\frac{\partial}{\partial t})=0
\end{equation*} hence $t_1\wedge\dots\wedge t_{n-1}\wedge t_k$ is also zero as a section of $f_*\omega_X$.

It remains to prove that given any choice of indices $\{j_1,\dots,j_{n-1}\}\subset \{1,\dots,k\}$ we have 
\begin{equation*}
\psi(t_{j_1}\wedge\dots\wedge t_{j_{n-1}}\wedge t_k\otimes\frac{\partial}{\partial t})=0.
\end{equation*}

It is enough to show this claim for $t_s\wedge t_2\wedge \dots\wedge t_{n-1}\wedge t_k$ where $s\neq 1$ since the same procedure can be iterated if necessary to obtain the general statement.

Using $t_1\wedge\dots\wedge t_{n-1}\wedge t_k=0$ and the induction hypothesis we have the inequalities
\begin{equation}
t_{s}\wedge t_2\wedge\dots\wedge t_{{n-1}}\wedge (t_k+t_1)=t_{s}\wedge t_2\wedge\dots\wedge t_{{n-1}}\wedge t_k=(t_{s}+t_1)\wedge t_2\wedge\dots\wedge t_{n-1}\wedge t_k
\end{equation}
Using the Massey triviality on the first equality we obtain
\begin{multline*}
\alpha_s\phi(\eta_2\wedge\dots\wedge\eta_{{n-1}}\wedge(\eta_k+\eta_1))+\sum^{n-1}_{l=2}\alpha_l\phi(\eta_{s}\wedge\eta_2\dots\wedge\widehat{\eta_{l}}\wedge\dots\wedge\eta_{{n-1}}\wedge(\eta_k+\eta_1))+\alpha_k \phi(\eta_{s}\wedge\eta_2\wedge\dots\wedge\eta_{{n-1}})=\\
=\beta_s\phi(\eta_2\wedge\dots\wedge\eta_{{n-1}}\wedge\eta_k)+\sum^{n-1}_{l=2}\beta_l\phi(\eta_{s}\wedge\eta_2\dots\wedge\widehat{\eta_{l}}\wedge\dots\wedge\eta_{{n-1}}\wedge\eta_k)+\beta_k\phi(\eta_{s}\wedge\eta_2\wedge\dots\wedge\eta_{{n-1}})
\end{multline*}
Using the second equality we have

\begin{multline*}
\beta_s\phi(\eta_2\wedge\dots\wedge\eta_{{n-1}}\wedge\eta_k)+\sum^{n-1}_{l=2}\beta_l\phi(\eta_{s}\wedge\eta_2\dots\wedge\widehat{\eta_{l}}\wedge\dots\wedge\eta_{{n-1}}\wedge\eta_k)+\beta_k\phi(\eta_{s}\wedge\eta_2\wedge\dots\wedge\eta_{{n-1}})=\\
=\mu_s\phi(\eta_2\wedge\dots\wedge\eta_{{n-1}}\wedge\eta_k)+\sum^{n-1}_{l=2}\mu_l\phi((\eta_{s}+\eta_1)\wedge\eta_2\dots\wedge\widehat{\eta_{l}}\wedge\dots\wedge\eta_{{n-1}}\wedge\eta_k)+\mu_k \phi((\eta_{s}+\eta_1)\wedge\eta_2\wedge\dots\wedge\eta_{{n-1}})
\end{multline*}
Where the $\alpha$'s, $\beta$'s and $\mu$'s are holomorphic functions on $A$.
Using the strictness as before we immediately find that all these functions vanish, giving us the desired result.

As a final remark note that the uniqueness of the liftings $t_i$ is immediate.
\end{proof}
In the case $W\subset \Gamma(A,\mD)\subset\Gamma(A,K_\partial)$, Proposition \ref{wedge0} becomes  
\begin{prop}
\label{chiusezero}
Let $W\subset \Gamma(A, \mD)$ a Massey trivial strict subspace. Then there exists a unique $\widetilde{W}\subset \Gamma(A,f_*\Omega^1_{S,d})$ lifting $W$ such that the map
$$
\bigwedge^{n}\widetilde{W}\to \Gamma(A,f_*\omega_{X})
$$ is zero.
\end{prop}
\begin{proof}
We only have to verify that $\widetilde{W}\subset \Gamma(A,f_*\Omega^1_{S,d})$.
This holds because by definition of $\mD$ all the liftings of the sections $\eta_i$ are closed forms.
\end{proof}

The following result relates in some sense the property of being locally Massey trivial and globally Massey trivial.

\begin{prop}
\label{lglob}
Let $W\subset \Gamma(B, K_{\partial})$ be a strict subspace of global sections of $K_{\partial}$ and let $A\subset B$ an open contractible subset. If the sections of $W$ are Massey trivial when restricted to $A$ then they are Massey trivial everywhere. 
\end{prop}
\begin{proof}
Since the dimension of $W$ doesn't play any role, for simplicity we assume that $\dim W=n$. Take $\eta_1,\ldots,\eta_{n}$ a basis of $W$ and $\widetilde{W}=\langle s_1,\dots,s_{n}\rangle < \Gamma(B,f_*\Omega^1_X)$ a lifting of $W$ via Lemma (\ref{split}). The $s_i$'s are global holomorphic 1-forms on $X$. 

By hypothesis there exist $t$ local parameter and $a_i$ holomorphic functions on $A$ such that
\begin{equation}
\label{ipotesi1}
\psi(s_1\wedge\dots\wedge s_{n}\otimes \frac{\partial}{\partial t})=\sum^{n}_{i=1} a_i\omega_i
\end{equation}
Now take $A'\subset B$ another open contractible subset with nontrivial intersection with $A$ and call
$$
\psi'\colon \widetilde{W}\otimes T_B\to f_*\omega_{X/B_{|A'}}
$$ the map which gives the Massey product on $A'$.
By the strictness hypothesis we can complete the $\omega_i$'s to a local frame of $f_*\omega_{X/B_{|A'}}$, hence in this frame we have
$$
\psi'(s_1\wedge\dots\wedge s_{n}\otimes \frac{\partial}{\partial t'})=\sum^{n}_{i=1} a'_i\omega_i+\sum_{j} b_j\tau_j.
$$
By comparing with (\ref{ipotesi1}) on $A\cap A'$ where $\partial/\partial t'=\partial t/\partial t'\cdot \partial/\partial t$ we have 
\begin{multline}
\sum^{n}_{i=1} a'_i\omega_i+\sum_{j} b_j\tau_j=\psi'(s_1\wedge\dots\wedge s_{n}\otimes \frac{\partial}{\partial t'})=\\=\frac{\partial t}{\partial t'}\psi(s_1\wedge\dots\wedge s_{n}\otimes \frac{\partial}{\partial t'})=\sum^{n}_{i=1} \frac{\partial t}{\partial t'}a_i\omega_i.
\end{multline}
hence the $b_j$ vanish on $A\cap A'$ and hence everywhere on $A'$. This proves the Massey triviality on $A'$.
By iterating on an appropriate cover of $B$ we are done.  
\end{proof}
By the same computations we also immediately prove 
\begin{prop}
\label{locglob}
Under the hypotheses of the previous Proposition, call  $\widetilde{W}$ a lifting of $W$ in $H^0(B,f_*\Omega^1_X)$. We have that the image of $\bigwedge^{n}\widetilde{W} \to H^0(B,f_*\omega_X)$ is contained in the image of $H^0(B,\omega_B)\otimes \bigwedge^{n-1}\widetilde{W}\to H^0(B,f_*\omega_X)$. In particular if $B=\mP^1$ we have that the image of $\bigwedge^{n}\widetilde{W} \to H^0(B,f_*\omega_X)$ is zero.
\end{prop}
\begin{proof}
We take $\eta_i$ and $s_i$ as in the proof of the previous Proposition.
By the same computations we get on $A\cap A'$ the relation 
$$
a_i'=\frac{\partial t}{\partial t'}a_i
$$ hence the 1-forms on $B$ defined by $a_idt$ and $a_i'dt'$ glue together and produce global 1-forms on $B$, call them $\sigma_i$. From the Massey triviality we then obtain
$$
s_1\wedge\dots\wedge s_n=\sum_{i=1}^n \sigma_i\wedge s_1\wedge\dots\wedge\widehat{s_i}\wedge\dots\wedge s_n
$$ which is our thesis.
\end{proof}

These results allow us to give a \lq\lq local to global\rq\rq version of Proposition \ref{wedge0} and \ref{chiusezero}

\begin{prop}
\label{localglobal}
Let $W\subset \Gamma(B, K_{\partial})$ a strict subspace of global sections of $K_{\partial}$ and let $A\subset B$ be an open contractible subset. If the sections of $W$ are Massey trivial when restricted to $A$ then there exist a unique lifting $\widetilde{W}\subset \Gamma(B, f_*\Omega^1_X)$ such that 
$$
\bigwedge^{n}\widetilde{W}\to \Gamma(B,f_*\omega_{X})
$$ is zero. If furthermore $W\subset \Gamma(B, \mD)$ then $\widetilde{W}\subset \Gamma(B, f_*\Omega^1_{X,d})$.
\end{prop}

\begin{rmk}
The second statement is meaningful only if $B$ is not compact, otherwise global 1-forms on $X$ are automatically closed. 
\end{rmk}

\begin{rmk}
Given an $n$-dimensional variety $Y$, a subspace $W$ of $H^0(Y,\Omega^1_Y)$ is usually called strict if the natural map from $\bigwedge^n W$ to $H^0(Y,\omega_Y)$ is an isomorphism on the image. See \cite[Definition 2.1 and 2.2]{Ca2} and \cite[Definition 2.2.1]{RZ1}.

If $W$ as above is a subspace of sections of $\mD$, $W\subset\Gamma(A,\mD)\subset\Gamma(A,K_\partial)$, then we can see $W$ as a subspace of $H^0(F_b,\Omega^1_{F_b})$ since $\mD$ is a local system. If $W$ is strict in the usual sense then it is strict according to Definition (\ref{strict}). In fact if $W$ is strict in the usual sense we have the injection $\bigwedge^{n-1}W\hookrightarrow \mD^{n-1}$ and taking the tensor product by $\sO_A$ we immediately get the desired injection $$\bigwedge^{n-1}W\otimes \sO_A\hookrightarrow \mD^{n-1}\otimes \sO_A\hookrightarrow f_*\omega_{X/B_{|A}}.$$
\end{rmk}

\section{A result on the monodromy of $\mD$ and $\mD^{n-1}$}
\label{sez4}
In this section we will give some results on the monodromy action associated to the local systems $\mD$ and $\mD^{n-1}$.
\subsection{The notion of Albanese primitive variety}
First we need to recall the Generalized version of the Castelnuovo-de Franchis Theorem, proved independently by Catanese in \cite[Theorem 1.14]{Ca2} and Ran in \cite[Prop II.1]{Ran}.
\begin{thm}
\label{cast1}
Let $X$ be an $n$-dimensional compact K\"ahler manifold and $w_1,\dots, w_{k+1} \in H^0 (X,\Omega^1_X)$ linearly independent 1-forms such that $w_1\wedge\dots\wedge w_{k+1}=0$ and that no collection of $k$ linearly independent forms in the span of $w_1,\dots, w_{k+1}$ wedges to zero.
Then there exists a holomorphic map $f\colon X\to Y$ over a normal variety $Y$ of dimension $\dim Y= k$ and such that $w_i\in f^*H^0(Y,\Omega^1_Y)$. Furthermore $Y$ is of Albanese general type.
\end{thm}
We recall the definition of Albanese general type which is also given by Catanese in \cite{Ca2}:
\begin{defn}
An irregular variety $Y$ of dimension $k$ is called of Albanese general type if its Albanese dimension $a$ equals $k$ and its irregularity is $q>k$.
\end{defn}
In the same paper, the author also introduces the notion of higher irrational pencil as follows:
\begin{defn}
A higher irrational pencil is a morphism with connected fibers $X\to Y$ with target a normal variety of Albanese general type.
A variety X admitting no higher irrational pencil is said to be Albanese primitive, or simply primitive.
\end{defn}

Note that higher irrational pencils can be seen as the higher dimensional analogues to fibrations over curves of genus $g\geq 2$.

Given these definitions, the Generalized Castelnuovo-de Franchis can be restated as follows:
\begin{thm}
Let $X$ be an $n$-dimensional smooth variety. $X$ is primitive if and only if the maps 
$$ 
\bigwedge^kH^0(X,\Omega^1_X)\to H^0(X,\Omega^k_X)
$$ are injective on decomposable elements, that is elements of the form $w_1\wedge\dots\wedge w_{k}$.
\end{thm}

\begin{rmk}
Note that it may very well be that these maps are not injective on linear combinations of decomposable elements.
\end{rmk}

For our purposes we will need a slightly different version of the Castelnuovo-de Franchis theorem:

\begin{thm}
\label{cas2}
Let $X$ as above and $w_1,\dots, w_l \in H^0 (X,\Omega^1_X)$ linearly independent 1-forms such that $w_{j_1}\wedge\dots\wedge w_{j_{k+1}}= 0$ for every $j_1,\dots,j_{k+1}$ and that no collection of $k$ linearly independent forms in the span of $w_1,\dots, w_{j_{k+1}}$ wedges to zero. Then there exists a holomorphic map $f\colon X\to Y$ over a normal variety $Y$ of dimension $\dim Y=k$ and such that $w_i\in f^*H^0(Y,\Omega^1_Y)$. Furthermore $Y$ is of general type.
\end{thm}
\begin{proof}
We can take $w_1,\dots, w_{k+1}$ and apply Theorem \ref{cast1} to prove the existence of $Y$, then we have to prove that $w_{k+2},\dots,w_{l}$ also are pullback of forms of $Y$.

Without loss of generality we deal with $w_{k+2}$. 
Let $U$ a sufficiently small open subset of $X$ where $w_1\wedge\dots\wedge w_{k}\neq 0$. From $w_1\wedge\dots\wedge w_{k}\wedge w_{k+2}=0$ we obtain that $w_{k+2}=\sum_{i=1}^k f_iw_i$ with $f_i$ holomorphic functions on $U$. It is also clear that the $f_i$'s are pullback of holomorphic functions on $Y$, otherwise $dw_{k+2}$ would be different from zero which is not possible, $w_{k+2}$ being a holomorphic form.

From the equality 
$$
w_1\wedge\dots\wedge\widehat{w_i}\wedge\dots\wedge w_{k}\wedge w_{k+2}=(-1)^{k-i}f_iw_1\wedge\dots\wedge w_{k}
$$
we also get that $f_i$ are global meromorphic functions.

This discussion tells us that $w_{k+2}=f^*\alpha$ is the pullback of a meromorphic 1-form on $Y$, but since $w_{k+2}$ is holomorphic and the pullback does not remove the poles of $\alpha$, it follows that $\alpha$ is already holomorphic on $Y$ and this is our thesis.   

To prove that $Y$ is of general type we use a standard argument based on a structure theorem by Ueno, see \cite[Theorem 10.9]{U} or \cite[Theorem 3.7]{Mo}. 
We know by Theorem \ref{cast1} that $Y$ is of Albanese general type, that is its Albanese map $\text{alb}\colon Y\to \text{Alb}(Y)$ is generically finite on its image $\tilde{Y}$ and not surjective. 
By the above mentioned theorem the Albanese image $\tilde{Y}$ is ruled by positive dimensional subtori over a basis Z which is a subvariety of general type of an abelian variety. Since $\dim Z<\dim Y=k$ a contradiction follows since it is possible to find $k$ linearly independent 1-forms \textit{in the span} of the $w_i$'s with vanishing wedge product.  
\end{proof}
\begin{rmk}
\label{bastachiuse}
It is not difficult to check that the proof of the Castelnuovo-de Franchis also works if $X$ is not compact if we further assume that the 1-forms $w_i $ are closed.
\end{rmk}
\subsection{Massey trivial strictness and monodromy representations}

The theory developed in the previous section is obviously strictly related to the Castelnuovo-de Franchis theorem.

Let $W$ be a Massey trivial subspace of $\mD$ as in the previous section. Let $H$ be the kernel of the monodromy representation of $\mD$, which is a normal subgroup of $\pi_1(B, b)$, and call $H_W$ the subgroup of $H$ which acts trivially on W. For every subgroup $K<H_W$, we denote by $B_K\to B$ the \'{e}tale base change of group $K$ and by $X_K\to B_K$ the associated fibration. We have the following:
\begin{thm}
\label{castmassey}
Let $X\to B$ a semistable fibration with $\dim X=n$. Let $A\subset B$ be an open subset and $W\subset \Gamma(A,\mD)$ a Massey trivial strict subspace.
Then $X_K$ has a higher irrational pencil $h_K\colon X_K\to Y$ over a normal $(n-1)$-dimensional variety of general type $Y$ such that $W\subset h_K^*(H^0(Y,\Omega^1_Y))$. Furthermore if $W$ is a maximal Massey trivial strict subspace we have the equality $W= h_K^*(H^0(Y,\Omega^1_Y))$.
\end{thm}
\begin{proof}
Since the proof is the same for every $K$, we prove the statement for $K=H_W$.
 
The key point is that since $H_W$ acts trivially on $W$, with the base change $B_{H_W}\to B$ the sections of $W$ give raise to global sections on $B_{H_W}$. More precisely call $\mD_W$ the local system on $B_{H_W}$ obtained by inverse image of $\mD$, then $W$ extends to global sections in $\Gamma(B_{H_W},\mD_W)$.
We can then apply Proposition \ref{localglobal} and find global 1-forms of $X_W$ which satisfy the hypotheses of the Castelnuovo-de Franchis Theorem \ref{cas2}.
Note that even in the case $X_{H_W}$ not compact, the proof works because by Proposition \ref{localglobal} we find closed 1-forms and the Castelnuovo-de Franchis still works as pointed out in Remark \ref{bastachiuse}.
\end{proof}
\subsection{Closures under the monodromy action}
A subspace $W\subset \Gamma(A,\mD)$ as in the above theorem naturally generates a local system in $\mD$ by taking the closure of $W$ under the monodromy action.
More precisely, denote by $\rho$ the monodromy map associated with $\mD$ and by $G=\pi_1(B, b)/\ker \rho$ the monodromy group acting non-trivially on $\mD$. 
The local system generated by $W$ is by definition the local system with stalk $\widehat{W}=\sum_{g\in G}g\cdot W$. We will denote it by $\mW$.
\begin{defn}
\label{mastrivgen}
If $W$ is Massey trivial, we will say that $\mW$ is Massey trivial generated.
\end{defn} See \cite[Definition 5.5]{PT}.

To the local system $\mW$ we associate its monodromy group as follows. The action $\rho$ of the fundamental group $\pi_1(B, b)$ on the stalk of $\mD$ restricts to an action $\rho_W$ on the stalk of $\mW$, that is 
$$
\rho_\mW\colon \pi_1(B, b)\to \text{Aut}(\widehat{W}).
$$

We will now construct an action of the monodromy group $G_\mW=\pi_1(B, b)/\ker \rho_\mW\cong \Ima \rho_\mW$ on a suitable set which will allow us to study this group. Sometimes we will denote $\ker \rho_\mW$ by $H_\mW$, so that $G_\mW=\pi_1(B,b)/H_\mW$.

Let $u_\mW\colon B_\mW\to B$ the covering classified by the subgroup $H_\mW$ and $f_\mW\colon X_\mW\to B_\mW$ the associated fibration. In a similar way to that seen before, the inverse image of the local system $\mW$ on $B_\mW$ is trivial and we will often identify the sections of $W$ and their unique liftings provided by Proposition \ref{localglobal}, which are global closed 1-forms on $X_\mW$.

By Theorem \ref{castmassey} applied to the subgroup $H_\mW$ of $H_W$ we get a map $h\colon X_{\mW}\to Y$ which can be composed with the action of $G_\mW$ on $X_\mW$ obtained from the standard action of $G_\mW$ on $B_\mW$. We call $h_g$ the composition.

\begin{equation}
\xymatrix{
X_\mW\ar^{g}[r]\ar_{h_g}[rrd]&X_\mW\ar^{h}[rd]&\\
&&Y
}
\end{equation}
Now we take a regular point $b\in A\subset B$ and a preimage $b_0$ of $b$ via $B_\mW\to B$. Denote by $F_0$ the fibre of $f_\mW$ over $b_0$ and consider the above diagram restricted to $F_0$.

\begin{equation}
\xymatrix{
F_0\ar@{^{(}->}[r]\ar_{k_g}[drrr]&X_\mW\ar^{g}[r]\ar^{h_g}[rrd]&X_\mW\ar^{h}[rd]&\\
&&&Y
}
\end{equation}

\begin{rmk}
The map $F_0\to Y$ is surjective since the pullback of the 1-forms of $Y$ gives on $F_0$ the subspace $W$ which we recall is strict by hypothesis.
\end{rmk}

We can define two sets of functions 
$$
\mathcal{H}=\{h_g\colon X_\mW\to Y\mid g\in G_\mW\}
$$ and
$$
\mathcal{K}=\{k_g\colon F_0\to Y\mid g\in G_\mW\}
$$ and $G_\mW$ acts naturally on both by $g_1\cdot h_{g_2}=h_{g_2g_1}$ and $g_1\cdot k_{g_2}=k_{g_2g_1}$. We will focus in particular on the action $G_\mW\times \sK\to \sK$ and prove that it is faithful.
We start with the following lemma see \cite[Lemma 6.1]{PT}.
\begin{lem}
\label{formula1}
Let $e$ be the neutral element of $G_\mW$ and $\alpha\in H^0(Y,\Omega^1_Y)$. Then for each $g\in G_\mW$,
\begin{equation}
k_g^*(\alpha)=g^{-1}k_e^*(\alpha)
\label{formula}
\end{equation}
where $g^{-1}$ acts on $k_e^*(\alpha)\in W$ via the monodromy action defining $\mW$.
\end{lem}
\begin{proof}
Let $\beta=h_e^*(\alpha)$ the global closed 1-form in $X_\mW$ obtained by pullback. Clearly 
$$
k_g^*(\alpha)=(g^*h_e^*(\alpha))_{|F}=g^*\beta_{|F}=\beta_{|F_{g^{-1}{b}}}
$$
On the other hand
$$
g^{-1}k_e^*(\alpha)=g^{-1}\beta_{|F}=\beta_{|F_{g^{-1}{b}}}.
$$ 
\end{proof}

\begin{lem}
\label{faithful}
The action of $G_\mW$ on $\sK$ is faithful.
\end{lem}
\begin{proof}
Take $g\in G_\mW$, $g\neq e$ and we prove that there exist an element $k_{g'}$ of $\sK$ such that $g\cdot k_{g'}\neq k_{g'}$. Since by definition of the action we have $g\cdot k_{g'}=k_{g'g}$ we have to prove that 
$$
k_{g'}\colon F_0\to Y \quad\text{and}\quad k_{g'g}\colon F_0\to Y
$$ are different morphisms. We will prove this statement at the level of 1-forms, more precisely we prove that 
$$
k_{g'}^*\colon H^0(Y,\Omega^1_Y)\to H^0(F_0,\Omega^1_{F_0}) \quad\text{and}\quad k_{g'g}^*\colon H^0(Y,\Omega^1_Y)\to H^0(F_0,\Omega^1_{F_0})
$$ are different.

Now since $g\neq e$, there exist an element of $\mW$ which is not fixed by $g$, and, since $\mW$ is the local system generated by $W$, we can assume that this element is of the form $\hat{g}w$ with $w\in W$ and $\hat{g}\in G$, that is $g\hat{g}w\neq \hat{g}w$.

 By Theorem \ref{castmassey}, $w=k_e^*(\alpha)$ for some $\alpha \in H^0(Y,\Omega^1_Y)$, and by the previous Lemma we obtain 
$$
k_{g'}^*(\alpha)={g'}^{-1}k_e^*(\alpha)={g'}^{-1}w
$$ and
$$
k_{g'g}^*(\alpha)=({g'g})^{-1}k_e^*(\alpha)=({g'g})^{-1}w=g^{-1}g'^{-1}w
$$
Hence the thesis follows by taking $g'=\hat{g}^{-1}$.
\end{proof}
\subsection{Massey trivial generation and finiteness of monodromy groups}
We can now state the main theorem of this section:

\begin{thm}
\label{monfin}
Let $f \colon X \to B$ be a semistable fibration on a smooth projective curve B and let $\mW<\mD$ be a strict Massey trivial generated local system.
Then the associated monodromy group $G_\mW$ is finite and the fiber of $\mW$ is isomorphic to 
$$
\sum_{g\in G_\mW} k_{g}^*H^0(Y,\Omega^1_Y).
$$
\end{thm}
\begin{proof}
By the previous result Lemma \ref{faithful}, we have an inclusion 
$$
G_\mW\hookrightarrow \text{Aut}(\sK)
$$ hence it is enough to show that $\sK$ is a finite set. $\sK$ is contained in $\text{Mor}(F_0,Y)$ the set of all surjective morphisms from $F_0$ to $Y$ and this is finite being $Y$ of general type, see for example \cite[Theorem 1]{kob}

For the result on the stalk of $\mW$, recall that this stalk is $\widehat{W}=\sum_{g\in G}g\cdot W$. Hence we have
$$
\widehat{W}=\sum_{g\in G}g\cdot W=\sum_{g\in G_\mW}g\cdot W=\sum_{g\in G_\mW} k_{g}^*H^0(Y,\Omega^1_Y)
$$where the second equality comes from the fact that $H_\mW$ fixes $W$ and the last from Lemma \ref{formula1}.
\end{proof}

Now if $\mD$ itself is Massey trivial generated we have the immediate corollary:
\begin{cor}
\label{fin2}
If $\mD$ is Massey trivial generated by a strict subspace, then the monodromy group $G$ is finite.
If furthermore the map $\bigwedge^{n-1}\mD\to \mD^{n-1}$ is surjective, the local system $\mD^{n-1}$ also has finite monodromy. 
\end{cor}
\begin{proof}
The first statement is immediate by Theorem \ref{monfin}. 

For the second it is enough to note that if $\bigwedge^{n-1}\mD\to \mD^{n-1}$ is a surjective map of local systems, then the monodromy group of $\mD^{n-1}$ is a subgroup of the monodromy group of $\bigwedge^{n-1}\mD$ and latter is finite by the first statement.
\end{proof}

To give an example where $\mD$ is Massey trivial generated, we can extend to any dimension the case of hyperelliptic fibration studied in \cite{PT}:

\begin{prop}
\label{hyper}
Let $X\to B$ a semistable fibration such that the general fiber $F$ has odd dimension and has an involution $\sigma$ such that $F/\sigma$ has $p_g=0$.
If $\mD$ is generated by anti-invariant 1-forms, then $\mD$ is Massey trivial generated.
\end{prop}
\begin{proof}
Take $\eta_1,\dots,\eta_n\in \mD_b\subset H^0(F,\Omega^1_F)$ and construct their Massey product $m_{\xi_b}(\eta_1,\dots,\eta_n)$. By the fact $\dim F=n-1$ is odd, we get that the Massey product is antisymmetric.

The involution $\sigma$ gives a map $\sigma^*\colon H^0(\omega_F)\to H^0(\omega_F)$ which is the multiplication by $-1$. Hence we get
$$
\sigma^*m_{\xi_b}(\eta_1,\dots,\eta_n)=-m_{\xi_b}(\eta_1,\dots,\eta_n)
$$ while at the same time
$$
\sigma^*m_{\xi_b}(\eta_1,\dots,\eta_n)=m_{\xi_b}(-\eta_1,\dots,-\eta_n)
$$ by the fact that the $\eta_i$'s are anti-invariant.
Since $n$ is even, this gives the desired vanishing of the Massey product.
\end{proof}

To close this section we recall that the finiteness of the monodromy group of a local system is equivalent to the semi-ampleness of the unitary flat vector bundle, see for example \cite[Theorem 2.5]{CD1}.
\begin{cor}
Under the hypotheses of Corollary \ref{fin2}, the unitary flat bundle of the second Fujita decomposition is semi-ample.
\end{cor}
\begin{proof}
We have that $\mD^{n-1}$ has finite monodromy and recall that it is the local system associated to the second Fujita decomposition by Theorem \ref{fujitaii}.
\end{proof}

\section{The adjoint local system}
\label{sectionclosed}

We introduce now another local system which in some sense embodies the notion of Massey product of sections of $\mD$.

To do this first note that the main ingredient for the construction introduced in Section \ref{sezioneaggiunta} is the following.
We take $\mD$ and choose a lifting via the splitting of
\begin{equation}
\xymatrix{
0\ar[r]& \omega_B\ar[r]&f_*\Omega^1_{X,d}\ar[r]& \mD\ar[r]\ar @/_1.3pc/ [l]_{} & 0
}
\end{equation}
which is guaranteed by Lemma (\ref{split}) and Lemma (\ref{lemmainclusione}).
After choosing such a lifting $\widehat{\mD}\hookrightarrow f_*\Omega^1_{X,d}$, to construct Massey products we consider the image of
\begin{equation}
\bigwedge^n\mD\otimes T_B\to \bigwedge^n\widehat{\mD}\otimes T_B\to f_*\omega_X\otimes T_B\cong f_*\omega_{X/B}.
\end{equation}

Note that the lifting of $\mD$ is not unique in general, and we consider all the other liftings differing from $\widehat{\mD}$ by a global section of $H^0(B,\omega_B)$. Call $\widehat{\mD}_i$, $i\in I$ a certain set of indices, all these possible liftings. Since obviously
$$
\mD\cong \widehat{\mD}_i
$$ all the $\widehat{\mD}_i$ are local systems inside the sheaf $f_*\Omega^1_{X,d}$. It makes sense to consider the local system generated by the $\widehat{\mD}_i$ which we will denote by $\langle \widehat{\mD}_i\rangle$. We then get a short exact sequence of local systems
\begin{equation}
0\to \mL\to \langle \widehat{\mD}_i\rangle\to \mD\to 0
\label{seqlocsys}
\end{equation} of all the possible liftings of $\mD$.
$\mL$ is the local system generated by the global sections of $\omega_B$.
Denote by $\mH_i$ the image of $\bigwedge^n \widehat{\mD}_i$ via the map 
$$
\bigwedge^n f_*\Omega^1_{X,d}\to f_*\omega_X.
$$
Now take $n$ sections $\eta_1,\dots,\eta_n$ in $\Gamma(A,\mD)$ and two different sets of liftings $s_1,\dots,s_n$ in $\Gamma(A,\widehat{\mD}_j)$ and $s'_1,\dots,s'_n$ in $\Gamma(A,\widehat{\mD}_k)$, with $j,k\in I$. If we want to construct the respective Massey products we have to take the wedge $s_1\wedge\dots\wedge s_n$ in $\Gamma(A,\mH_j)$ and $s'_1\wedge\dots\wedge s'_n$ in $\Gamma(A,\mH_k)$ and they differ by an element of $ \bigwedge^{n-1}\mD\otimes\mL$. More precisely denote by $\langle \mH_i\rangle$ the local system generated by the $\mH_i$'s and by $\widetilde{\mD}$ the image of 
$$
\bigwedge^{n-1}\mD\to \mD^{n-1}\hookrightarrow f_*\omega_{X/B}
$$ then we have an exact sequence 
\begin{equation}
\label{esattalocal}
0\to \widetilde{\mD}\otimes\mL\to \langle \mH_i\rangle\to \mH\to 0.
\end{equation}

$\mH$ can be seen as the local system on $B$ which contains all the Massey products obtained by sections of $\mD$ but without the ambiguity given by the choice of the liftings. 

We want now to give the analogue of Definition \ref{mtrivial} under this new light. Therefore take a point $b\in B$ and a vector subspace $W$ of $\mD_b$ generated by sections $\eta_1,\dots,\eta_n$. $W$ is not necessarily invariant under the monodromy action, but we can consider the local system $\mW$ generated by $W$, as we have done in the previous sections.
Since $\mW$ is a sublocal system of $\mD$ it makes sense to retrace the above construction for $\mW$ instead of the whole $\mD$. We end up with a local system $\mH_W$ inside $\mH$.

\begin{defn}
\label{agglocal}
We call $\mH_W$ the adjoint local system of $W$.
\end{defn}

Even if $\mH_W$ contains the Massey products of sections of W, this approach is different from the one given in Section \ref{sezioneaggiunta} and in particular in Definition \ref{mtrivial}.
In Definition \ref{mtrivial} we consider the Massey product up to the $\sO_B$-submodule generated by the sections $\eta_1\wedge\dots\wedge\widehat{\eta_i}\wedge\dots\wedge\eta_n$ while here we consider the Massey products in $\mH_W$ up to $\widetilde{\mD}\otimes\mL$ which is a local system and not a $\sO_B$-submodule.
This entails that we no longer have an analogue of Proposition \ref{chiusezero} in this setting. That is, even if $\mH_W$ is trivial we cannot say that there is a vector space lifting $W$ and such that the wedge of all its sections is zero. In fact in the proof of this result it is essential to change the liftings of $\eta_i$ by an arbitrary holomorphic form in $\omega_B$. Note that here this is not possible any more because we can change liftings only up to elements of $\mL$. 

On the other hand at the level of global forms the situation is very similar:
\begin{prop} 
If $W\subset \Gamma(B,\mD)$ is an $n$-dimensional subspace of global sections of $\mD$, then $W$ is Massey trivial according to Definition \ref{mastriv} if and only if $\mH_W=0$.
\end{prop}
\begin{proof}
If $W$ is Massey trivial according to Definition \ref{mastriv}, by Proposition \ref{locglob} we have that every section $s_1\wedge\dots\wedge s_n$ is in the image of $\widetilde{\mD}\otimes \mL$, hence $\mH_W=0$.
The other implication is even more immediate.
\end{proof}
Note that in the previous sections, for example when using the Castelnuovo-de Franchis theorem, we have always worked with global sections of $\mD$, maybe after passing to a covering of $B$. This means that for the purpose of this paper both the constructions have the same effectiveness.

The vanishing of $\mH_W$ also controls the following property of Massey products. As always take $n$ sections $\eta_1,\dots,\eta_n$ in $W$ and liftings $s_1,\dots,s_n$.
The following diagram tensored by $\omega_B$
\begin{equation}
\xymatrix{
&&f_*\omega_X\otimes T_B\\
f_*\Omega^{n-1}_{X,d}\ar@{->>}[r]\ar[urr]&\mD^{n-1}\ar@{^{(}->}[r]&f_*\omega_{X/B}\ar@{=}[u]
}
\end{equation}
 tells us that the Massey product $\psi(s_1\wedge\dots\wedge s_n\otimes \partial/\partial t)$, which is a section of $f_*\omega_{X/B}$, is  actually in $\mD^{n-1}\hookrightarrow f_*\omega_{X/B}$ if and only if the wedge $s_1\wedge\dots\wedge s_n$ is in the image of
\begin{equation}
\label{aggiuntachiusa}
f_*\Omega^{n-1}_{X,d}\otimes \omega_B\to \omega_X.
\end{equation}
The local system $\widetilde{\mD}\otimes\mL$ has this property, hence we can formulate the following
\begin{prop}
\label{chiusasoll}
Take $W$ as above. If $\mH_W$ is zero, all the Massey products constructed starting by sections of $W$ are in $\mD^{n-1}$, that is, they can be lifted to closed differential forms of the ambient variety $X$.
\end{prop}

\section{Massey products and first Fujita decomposition}
\label{sez6}

In the previous sections we have shown the relation between Massey products and the second Fujita decomposition, that is we have considered Massey products starting from sections of the local system $\mD$ and their relation with the local system $\mD^{n-1}$, which we recall gives $\sU$ when tensored by $\sO_B$, see Theorem \ref{secfujita}.  

We want now to briefly look at the situation considering the first Fujita decomposition. The difference with the previous case is that we will no longer deal with local systems and unitary flat vector bundle associated, but with trivial vector bundles.

Recall the first Fujita decomposition
\begin{equation}
f_*\omega_{X/B}\cong \sO_B^h\oplus \sE
\end{equation} with $\sE$ a locally free nef sheaf on $B$ with $h^1(B,\sE\otimes \omega_B)=0$.
%

Take the exact sequence 
$$
0\to \omega_B\to f_*\Omega^1_X\to f_*\Omega^1_{X/B}\to \dots
$$
and call $V$ the vector space 
\begin{equation}
V:=H^0(X,\Omega^1_X)/f^*H^0(B,\omega_B).
\end{equation} $V$ injects in $H^0(F_b,\Omega^1_{F_b})$ for all $b\in B^0$ and furthermore we have the following inclusions 
$$
V\subset \mD_b\subset K_\partial\otimes \mC(b)=\ker\delta_{\xi_b}\subset H^0(F_b,\Omega^1_{F_b}).
$$
The inclusion $\mD_b\subset K_\partial\otimes \mC(b)$ is the content of Lemma \ref{lemmainclusione}.
To justify the inclusion $V\subset \mD_b$, first recall that by definition $\mD_b$ is the subspace of sections of $H^0(F_b,\Omega^1_{F_b})$ which can be lifted to closed holomorphic forms on the variety $X$. Now since $X$ is compact, every global holomorphic form on $X$ is automatically closed and it immediately follows that $V$ is contained in $\mD_b$. 

Defining $V_B=V\otimes_\mC \sO_B$ we have the inclusion of vector bundles on $V_B\hookrightarrow \mD\otimes_\mC\sO_B\hookrightarrow K_\partial$ and therefore we can construct Massey products starting from sections of $V_B$.

Now call $H=H^0(X,\Omega^{n-1}_X)$ and $N=\{\phi\in H\mid \phi_{|F_b}=0, \forall b\in B^0\}$. By the exact sequence 
\begin{equation}
0\to \omega_B\otimes f_*\Omega^{n-2}_{X/B}\to f_*\Omega^{n-1}_X\to f_*\Omega^{n-1}_{X/B}\to \dots, 
\end{equation}
we have that $N=H^0(B,\omega_B\otimes f_*\Omega^{n-2}_{X/B})$.
Following the proof of \cite[Theorem 3.1]{Fu} we have that $H/N$ is an $h$-dimensional vector space and
$H/N\otimes_\mC\sO_B$ is the trivial vector bundle appearing in the first Fujita decomposition.
Alternatively this can be seen as follows. The wedge product gives a map
$$
H\otimes \omega_B\to f_*\omega_X
$$ and since $N\otimes \omega_B$ is in the kernel of such a map, we have 
$$
H/N\otimes \omega_B\to f_*\omega_X.
$$ This map is injective over $B^0$, hence it is injective everywhere. 
Now taking the tensor product by $\omega_B^\vee$ we get the desired inclusion
$$
H/N\otimes\sO_B\hookrightarrow f_*\omega_X\otimes \omega_B^\vee=f_*\omega_{X/B}
$$ giving the first Fujita decomposition.

We point out that, like for $V_B\hookrightarrow \mD\otimes\sO_B$, we also have $H/N\otimes_\mC\sO_B\hookrightarrow \mD^{n-1}\otimes_\mC\sO_B=\sU$.

The above discussion shows that we can construct the Massey product of sections of $V_B$ following the same recipe of Section \ref{sezioneaggiunta}.

Regarding Section \ref{sectionclosed}, we have that Sequence (\ref{seqlocsys}) for $V$ is exactly the sequence defining $V$
$$
0\to H^0(B,\omega_B)\to H^0(X,\Omega^1_X)\to V\to 0. 
$$
Furthermore denote by $U_n$ the image of $\bigwedge^n H^0(X,\Omega^1_X)$ in $H^0(X,\Omega^n_X)$ and by $U_{n-1}$ the image of $\bigwedge^{n-1} H^0(X,\Omega^1_X)$ in $H^0(X,\Omega^{n-1}_X)$. Then the local system $\mH$ previously introduced turns out to be the constant sheaf of stalk $U_n/(U_{n-1}\otimes H^0(B,\omega_B))$.

Note also that if we take a vector space of sections $W$ as in Section \ref{sectionclosed}, but with the extra assumption $W\subset V$, then $\mW=W$ since $V$ comes from global sections and hence it is invariant under monodromy.

It is not difficult to see that Proposition \ref{chiusasoll} now is 
\begin{prop}
Take $W\subset V$ as above. If $\mH_W$ is zero, all the Massey products constructed starting by sections of $W$ are in $H/N$, that is, they can be lifted to closed global differential forms of the ambient variety $X$. Furthermore if $\dim W=n$ then $\mH_W=0$ implies that the Massey product of sections of $W$ is trivial.
\end{prop}

\section{Sheaves generated by global sections and fibrations over $\mP^1$}
\label{sez7}
Consider the first Fujita decomposition of $f_*\omega_{X/B}$
\begin{equation}
f_*\omega_{X/B}\cong \sO_B^h\oplus \sE
\end{equation} with $\sE$ a locally free nef sheaf on $B$ with $h^1(B,\sE\otimes \omega_B)=0$.
In this section we generalize a result of Konno which gives an interesting bound for $h$, see \cite[Section 1]{K}.

We start by noticing that $f_*\omega_X$ is a locally free sheaf on $B$ of rank $p_g(F)$ with $F$ a general fiber of $f\colon X\to B$. Furthermore $f_*\omega_X\otimes \omega_B^\vee$ is nef by the work of Fujita \cite{Fu}.

Following Konno, we denote by $\sG$ the locally free subsheaf of $f_*\omega_X$ generated by the global sections and by $\sG'$ the locally free quotient $f_*\omega_X/\sG$. On the other hand $\sG^*$ will denote the sheaf $\sG^\vee\otimes \omega_B$.

Now let $\sF'$ the locally free sheaf generated by the global sections of $\sG^*$ and by $\sF=\sF'^*=\sF'^\vee\otimes\omega_B$. $\sF$ and $\sF^*$ are both generated by global sections and nef.

By definition we have $h^0(\sG)=h^0(f_*\omega_X)=p_g(X)$ and $h^0(\sG^*)=h^1(\sG)=h^0(\sF^*)=h^1(\sF)$.
\begin{prop}
If $b$ is the genus of $B$ and $r=\rank\sG$, we have the inequality
\begin{equation}
h \leq\rank \sF-(b-1)(p_g(F)-r)
\end{equation}
If the equality holds, then $\deg \sG'=2(b-1)(p_g(F)-r)$ and $\sF$ is a direct sum of $\rank \sF$ copies of $\omega_B$.
\end{prop}
\begin{proof}
By the work of Xiao, \cite[Page 477]{X}, we have the inequalities 
\begin{equation}
h^1(\sG)=h^1(\sF)\leq b\cdot\rank\sF-\frac{1}{2}\deg \sF
\label{xiao1}
\end{equation}
\begin{equation}
\deg \sG'\geq 2(b-1)(p_g(F)-r)
\label{xiao2}
\end{equation}
\begin{equation}
\deg \sF+\deg\sG'\geq 2(b-1)(p_g(F)-r+\rank\sF)
\label{xiao3}
\end{equation}
Now $h=h^1(f_*\omega_X)$ by Fujita and since $H^0(\sG)=h^0(f_*\omega_X)$ from the local sequence 
$$
0\to \sG\to f_*\omega_X\to\sG'\to 0 
$$ we obtain that $h=h^1(\sG)-\chi(\sG')$. By Riemann-Roch and by (\ref{xiao2}) we have $$\chi(\sG')=\deg\sG'-(b-1)(p_g(F)-r)\geq \deg \sG'/2,$$ hence 
$$
h=h^1(\sG)-\chi(\sG')\leq h^1(\sG)-\deg\sG'/2\leq b\cdot \rank\sF-(\deg\sF+\deg\sE')/2
$$ where the last inequality uses (\ref{xiao1}). Applying (\ref{xiao3}) gives the inequality in the statement.

If the equality holds then by reversing the previous computations we have that it must hold also in (\ref{xiao1}), (\ref{xiao2}) and (\ref{xiao3}). Putting those three together we obtain that in this case $h^1(\sF)=\rank\sF$, hence $h^0(\sF^*)=\rank\sF$. By the fact that $\sF^*$ and $\sF$ are both nef and generated by global sections, we have that $\sF^*$ is semi-stable of degree 0 and hence a sum of $\sO_B$. 
Tensoring by $\omega_B$ gives the result for $\sF$.
\end{proof}
\medskip
We can define the rank $r$ of a fibration $f\colon X\to B$ as the rank of the  the locally free subsheaf of $f_*\omega_X$ generated by the global sections.

\begin{thm} Let  $X$ be an irregular variety with $q(X)>n$. If $f\colon X\to\mP^1$ is a fibration with $r=p_g(F)$ then it is not of Albanese general type
\end{thm}
\begin{proof}
By Grothendieck's decomposition theorem and by Fujita decomposition theorem the assumption that $r=p_g(F)$ implies that $f_{*}\omega_X=\oplus_{j=0}^r\sO_{\mP^1}(a_j)$ where $a_j\geq 0$ and not necessarily distinct. Then in the second Fujita decomposition we have $\sU=0$. This implies that $\mathbb D^{n-1}=0$. This means that no $n$-dimensional subspace $W\subset H^0(X,\Omega^1_X)$ is strict.
\end{proof}


\begin{thebibliography}{Muk04}
%
%
\bibitem[BGN]{BGN}
M. \'A. Barja, V. Gonz\'alez-Alonso, J. C. Naranjo, \emph{Xiao’s conjecture for general fibred surfaces}, Journal f\"ur die reine und angewandte Mathematik 739 (2018), 297--308. 
%
%
    
%
%
%
%
%
\bibitem[Ca]{Ca2} F. Catanese, \emph{Moduli
and classification of irregular Kaehler manifolds (and algebraic varieties) with Albanese
general type fibrations},  Invent. Math. 104 (1991), no. 2, 263--289. 

\bibitem[CD1]{CD1}
F. Catanese, M. Dettweiler, \emph{Answer to a question by Fujita on Variation of Hodge Structures}, Higher Dimensional Algebraic Geometry: In honour of Professor Yujiro Kawamata's sixtieth birthday, Mathematical Society of Japan, Tokyo, Japan, (2017), 73--102.
\bibitem[CD2]{CD2}
F. Catanese, M. Dettweiler,  \emph{The direct image of the relative dualizing sheaf needs not be semiample}, C. R. Math. Acad. Sci. Paris 352 (2014), no. 3, 241--244.
\bibitem[CD3]{CD3}
F. Catanese, M. Dettweiler, \emph{Vector bundles on curves coming from variation of Hodge structures}, Internat. J. Math. 27 (2016), no. 7, 1640001, 25 pp.
\bibitem[CK]{CK}
F. Catanese, Y. Kawamata, \emph{Fujita decomposition over higher dimensional base}, European Journal of Mathematics 5,  (2019), 720--728.






%
%
%
%

\bibitem[CEZGT]{CEZGT} Eduardo Cattani, Fouad El Zein, Phillip A. Griffiths, and L\^e D\~ung Tr\'ang, editors.
{\em Hodge theory}, volume 49 of Mathematical Notes. Princeton University Press, Princeton,
NJ, 2014.

%
%
\bibitem[CNP]{CNP}
A. Collino, J. C. Naranjo, G. P. Pirola, {\em The Fano normal function}, J. Math. Pures Appl. (9) 98 (2012), no. 3, 346--366.
%
%
%
%

%

\bibitem[CP]{CP} A. Collino, G. P. Pirola, 
\emph{The Griffiths infinitesimal invariant for a curve in its Jacobian}, Duke Math. J., 78 (1995), no. 1, 59--88.

\bibitem[CRZ]{CRZ} L. Cesarano, L. Rizzi, F. Zucconi, \emph{On birationally trivial families and Adjoint quadrics}, submitted for publication, (2019).
%
%
\bibitem[De]{De}
P. Deligne, \emph{Equations Diff\'erentielles \`a Points Siunguliers R\'eguliers}, LNM 163, Springer-Verlag ,Berlin, Heidelberg, New York,  1970. 



\bibitem[Fu1]{Fu}
T. Fujita, \emph{On K\"ahler fiber spaces over curves}, 
 J. Math. Soc. Japan 30 (1978), no. 4, 779--794. 

\bibitem[Fu2]{Fu2}
T. Fujita, \emph{The sheaf of relative canonical forms of a K\"ahler fiber
space over a curve}, Proc. Japan Acad. Ser. A Math. Sci. 54 (1978),
no. 7, 183--184.

%
%
%
%
\bibitem[G-A]{victor}
V. Gonz\'alez-Alonso, \emph{On deformations of curves supported on rigid divisors}, Ann. Mat. Pura Appl. (4)
195(1), 111--132 (2016).
%
%
%
%
%


%

%
%
%
%
%
%
%

\bibitem[Il]{Il}
L. Illusie, {\em R\'eduction semi-stable et d\'ecomposition de complexes de de Rham  \`a coefficients}, Duke Math. J. 60 (1990), no. 1, 139--185.
%

\bibitem[KO]{kob}
S. Kobayashi, T. Ochiai, \emph{Meromorphic Mappings onto Compact Complex Spaces of General Type}, Inventiones mathematicae 31 (1976), 7--16. 


\bibitem[K]{K}
K. Konno, \emph{On the Irregularity of Special Non-Canonical Surfaces}, Publ. Res. Inst. Math. Sci., 30 (1994), no. 4,  671--688. 

%
%



\bibitem[MR]{MR}
 V. Maillot, D. R\"ossler {\em Une conjecture sur la torsion des classes de Chern des fibr\'es de Gauss-Manin}, Publ. Res. Inst. Math. Sci. 46 (2010), no. 4, 789--828.


\bibitem[Mo]{Mo} S. Mori, \emph{Classification of higher-dimensional varieties Proc}, Symp. Pure Math. 46 (1987), 269--332.
%
%
\bibitem[PP]{pp}
G. Pareschi, M. Popa, \emph{Strong generic vanishing and
a higher-dimensional Castelnuovo-de Franchis inequality}, Duke Math.
J., 150(2), 269--285, (2009).

\bibitem[PT]{PT}
G.P. Pirola, S. Torelli, \emph{Massey Products and Fujita decompositions on fibrations of curves}, Collectanea Mathematica, 71 (2020), 39--61. 

\bibitem[PS]{PS} C. A. M. Peters, J. H. M. Steenbrink, {\em Mixed Hodge structures}, Volume 52
of Ergebnisse der Mathematik und ihrer Grenzgebiete. 3. Folge. A Series of Modern
Surveys in Mathematics. Springer--Verlag, Berlin, 2008.

%
\bibitem[PR]{PR}
G. P. Pirola, C. Rizzi, {\em Infinitesimal invariant and vector bundles}, Nagoya Math. J., 186 (2007), 95--118.
%
%
\bibitem[PZ]{PZ}
G. P. Pirola, F. Zucconi, \emph{Variations of the Albanese morphisms,} J. Algebraic Geom., 12 (2003), no. 3, 535--572. 


\bibitem[Ran]{Ran}
Z. Ran \emph{On subvarieties of abelian varieties}, Invent. Math., 62(3) 459--479, (1981).


\bibitem[Ra]{Ra}
E. Raviolo, {\em Some geometric applications of the theory of variations of Hodge structures}, Ph.D. Thesis.
%
%

\bibitem[RZ1]{RZ1} L. Rizzi, F. Zucconi {\em Differential forms and quadrics of the canonical image}, Annali di Matematica Pura ed Applicata (2020). https://doi.org/10.1007/s10231-020-00971-w

 
\bibitem[RZ2]{RZ2} L. Rizzi, F. Zucconi {\em Generalized adjoint forms on algebraic varieties,}  Annali di Matematica Pura ed Applicata, Vol. 196, Issue 3, (2017), 819--836.
\bibitem[RZ3]{RZ3} L. Rizzi, F. Zucconi {\em On Green’s proof of infinitesimal Torelli theorem for hypersurfaces,} Vol. 29, Issue 4, (2018), 689--709
%

\bibitem[To]{To} S. Torelli, {\em Fujita decompositions and infinitesimal
invariants on fibred surfaces}, 2018, PhD Thesis.

%

\bibitem[U]{U}
K. Ueno, \emph{Classification theory of algebraic varieties and compact complex spaces}, Lecture
Notes in Math., vol. 439, Springer-Verlag, 1975.


\bibitem[Vo1]{Vo1}
C. Voisin, \emph{Hodge theory and complex algebraic geometry, I}. Translated from the French by Leila Schneps. Cambridge Studies in Advanced Mathematics, 76. Cambridge University Press, Cambridge, 2002.

\bibitem[Vo2]{Vo2}
C. Voisin, \emph{Hodge theory and complex algebraic geometry, II}. Translated from the French by Leila Schneps. Cambridge Studies in Advanced Mathematics, 77. Cambridge University Press, Cambridge, 2003.



\bibitem[X]{X}
 G. Xiao, \emph{Algebraic surfaces with high canonical degree}, Math. Ann., 274 (1986), 473--483.

\end{thebibliography}
\end{document}